\documentclass[11pt]{amsart}

\usepackage{graphicx}%
\usepackage{multirow}%
\usepackage{amsmath,amssymb,amsfonts}%
\usepackage{amsthm}%
\usepackage{mathrsfs}%
\usepackage[title]{appendix}%
\usepackage{xcolor}%
\usepackage{textcomp}%
\usepackage{manyfoot}%
\usepackage{booktabs}%
\usepackage{algorithm}%
\usepackage{algorithmicx}%
\usepackage{algpseudocode}%
\usepackage{listings}%
\usepackage[colorlinks=true, linkcolor=red, urlcolor=blue, citecolor=blue]{hyperref}

\theoremstyle{definition}
\newtheorem{theorem}{Theorem}[section]

\newtheorem{lemma}[theorem]{Lemma}
\newtheorem{corollary}[theorem]{Corollary}
\theoremstyle{definition}
\newtheorem{definition}[theorem]{Definition}

\theoremstyle{remark}
\newtheorem{remark}[theorem]{Remark}

\allowdisplaybreaks[1]
\usepackage[left=2.5cm, right=2.5cm, top=3cm, bottom=3cm]{geometry}

\title[Orlicz Function estimations of the $a$-Davis--Wielandt Radius in $C^*$-algebra]{On the Orlicz Function estimations of the $a$-Davis--Wielandt Radius in $C^*$-algebra}

\author[D. Bhattacharya, F. Kittaneh, A. Patra] {{Debarati Bhattacharya}$^{1}$, {Fuad Kittaneh}$^{2,3}$, {Arnab Patra}$^{4}$}

\address{$^{[1]}$ Department of Mathematics, Indian Institute of Technology Bhilai, Chhattisgarh, India 491002.}
\email{\url{debaratib@iitbhilai.ac.in}}

\address{$^{[2]}$ Department of Mathematics, The University of Jordan, Amman, Jordan.}
\email{\url{fkitt@ju.edu.jo}}

\address{$^{[3]}$ Department of Mathematics, Korea University, Seoul 02841, South Korea.}
\email{\url{fkitt@ju.edu.jo}}

\address{$^{[4]}$ Department of Mathematics, Indian Institute of Technology Bhilai,  Chhattisgarh, India 491002.}
\email{\url{arnabp@iitbhilai.ac.in}}

\subjclass[2020]{Primary 47A12, 46L05; Secondary 47A30}
\keywords{Davis--Wielandt radius, $a$-Davis--Wielandt radius, $C^*$-algebra, Inequality, Orlicz function}

\begin{document}




\begin{abstract}

We present a systematic study of $a$-Davis--Wielandt radius estimates for elements in a unital $C^*$-algebra through an Orlicz function approach. We derive novel bounds for the algebraic $a$-Davis--Wielandt radius and further obtain estimates for the classical algebraic Davis--Wielandt radius via polar decomposition and the Moore--Penrose inverse. The results obtained in this work extend, generalize, and unify a variety of established inequalities in the literature as special cases.
\end{abstract}

\maketitle

\section{Introduction}

The numerical range and numerical radius have long been fundamental objects of study in operator theory and $C^*$-algebras, owing to their rich geometric structure and numerous applications across functional analysis, Banach algebras, numerical analysis, and perturbation theory. A particularly significant generalization of the numerical range is the Davis--Wielandt shell and its corresponding radius. By capturing the joint geometric behavior of an operator (or algebra element) alongside its quadratic norm component, the Davis--Wielandt shell provides deeper structural insights, refined spectral bounds, and richer information regarding numerical stability than the classical numerical range alone. Consequently, extending these notions to algebraic, weighted, and semi-Hilbertian settings has become a rapidly growing and vibrant area of research in modern operator theory.

Let $\mathfrak{A}$ be a $C^*$-algebra with unit $\mathbf{1}$ and let $\mathfrak{A}^*$ denote its topological dual space. We denote by $\mathscr{B}(\mathscr{H})$ the $C^*$-algebra of all bounded linear operators on a complex Hilbert space $\mathscr{H}$, equipped with inner product $\langle\cdot,\cdot\rangle$ and induced norm $\|\cdot\|$. Let $\mathfrak{A}^+$ and $\mathscr{Z}(\mathfrak{A})$ be the cone of positive elements of $\mathfrak{A}$ and the center of $\mathfrak{A}$, respectively. A linear functional $f\in\mathfrak{A}^*$ is said to be positive if $f(x)\geq0$ for all $x\in\mathfrak{A}^+$. The set of all states on $\mathfrak{A}$ is denoted by $\mathfrak{S}(\mathfrak{A})$, which is the collection of all positive linear functionals $f$ on $\mathfrak{A}$ such that $f(\mathbf{1})=\|f\|=1$.

Throughout the article, $a$ is symbolized as a non-zero positive element of $\mathfrak{A}$. Following recent developments, we consider the set of $a$-states defined by $\mathfrak{S}_a(\mathfrak{A})=\{f\in\mathfrak{A}^*:f\geq0,\ f(a)=1\}$, which is simply the set $\mathfrak{S}_a(\mathfrak{A})=\left\{\frac{g}{g(a)}:g\in\mathfrak{S}(\mathfrak{A}),\ g(a)\neq0\right\}$. Note that the set $\mathfrak{S}_a(\mathfrak{A})$ is non-empty, convex, $w^*$-closed and $w^*$-compact if and only if $a$ is invertible in $\mathfrak{A}$; see \cite[Proposition~2.3]{bourhim2021a-numerical}. For an element $x\in\mathfrak{A}$, define $\|x\|_a=\sup\left\{\sqrt{f(x^*ax)}:f\in\mathfrak{S}_a(\mathfrak{A})\right\}$. Clearly, $\|\cdot\|_{\mathbf{1}}=\|\cdot\|$ and $\|x\|_a=0$ if and only if $ax=0$. It can also happen that $\|x\|_a=\infty$ for some $x\in\mathfrak{A}$, since $\mathfrak{S}_a(\mathfrak{A})$ is not always compact \cite[Example~3.2]{bourhim2021a-numerical}. In this context, we denote $\mathfrak{A}^a=\{x\in\mathfrak{A}:\|x\|_a<\infty\}$. It is observed that $\mathfrak{A}^a$ is a subalgebra of $\mathfrak{A}$, not necessarily closed, and further $\mathfrak{A}^a=\mathfrak{A}$ if $a\in\mathscr{Z}(\mathfrak{A})$. Moreover, by \cite[Proposition~3.3]{bourhim2021a-numerical}, $\|\cdot\|_a$ is a seminorm on $\mathfrak{A}^a$ satisfying $\|xy\|_a\leq\|x\|_a\|y\|_a$ for all $x,y\in\mathfrak{A}^a$.

For an element $x\in\mathfrak{A}$, an element $x^{\sharp_a}\in\mathfrak{A}$ is said to be an $a$-adjoint of $x$ if $ax^{\sharp_a}=x^*a$. We denote by $\mathfrak{A}_a$ the set of all elements of $\mathfrak{A}$ that admit $a$-adjoints. In addition, $\mathfrak{A}_a$ forms a subalgebra of $\mathfrak{A}$ but is neither closed nor dense in $\mathfrak{A}$. In particular, $\mathfrak{A}_a=\mathfrak{A}$ if $\mathfrak{A}$ is commutative. For $x\in\mathfrak{A}$, the existence and uniqueness of the $a$-adjoint elements are not guaranteed, in general. However, if $x\in\mathfrak{A}_a$ and $x^{\sharp_a}$ is an $a$-adjoint of $x$, then, by \cite[Corollary~4.9]{bourhim2021a-numerical}, $\|x\|_a^2=\|xx^{\sharp_a}\|_a=\|x^{\sharp_a}x\|_a=\|x^{\sharp_a}\|_a^2.$ An element $x\in\mathfrak{A}$ is said to be $a$-self-adjoint if $ax$ is self-adjoint, that is, $ax=x^*a$. Similarly, $x$ is called $a$-positive whenever $ax$ is positive. Note that an $a$-positive element is always $a$-self-adjoint, and the elements $x^{\sharp_a}x$, $xx^{\sharp_a}$ are both $a$-positive. Every element of $\mathfrak{A}_a$ can be written as the sum of two $a$-self-adjoint elements. Unlike the classical Cartesian decomposition in a $C^*$-algebra, such a decomposition does not need to be unique. If $x^{\sharp_a}$ is an $a$-adjoint of $x$, then $x=\Re(x)+i\Im(x)$, where $\Re(x)=\frac{x+x^{\sharp_a}}{2}$ and $\Im(x)=\frac{x-x^{\sharp_a}}{2i}$.

The \emph{$a$-numerical range} and the \emph{$a$-numerical radius} of an element $x\in\mathfrak{A}$ are defined, respectively, by $V_a(x)=\{f(ax):f\in\mathfrak{S}_a(\mathfrak{A})\}$ and $v_a(x)=\sup\{|w|:w\in V_a(x)\}$. Unlike the classical algebraic numerical range, the algebraic $a$-numerical range $V_a(x)$ is not necessarily closed or bounded. Several fundamental properties of $V_a(x)$ and $v_a(x)$ have been established in, for example, \cite{alahmari2022,bourhim2021a-numerical,mabrouk2023}.

Motivated by developments in operator geometry, the \emph{algebraic Davis--Wielandt shell} and the corresponding \emph{Davis--Wielandt radius} for unital $C^*$-algebras were recently investigated in \cite{arambavsic2018roberts,bhattacharya2026geometry}. As a natural extension of the algebraic Davis--Wielandt shell, the algebraic $a$-Davis--Wielandt shell of an element $x\in\mathfrak{A}$ was introduced in \cite{dehghani2025roberts}. The \emph{algebraic $a$-Davis--Wielandt shell} $DV_a(x)$ and the corresponding \emph{algebraic $a$-Davis--Wielandt radius} $dv_a(x)$ are defined, respectively, by
\[
DV_{a}(x)=\left\{(f(ax),f(x^*ax)):f\in\mathfrak{S}_a(\mathfrak{A})\right\},\quad
dv_a(x)=\sup\left\{\sqrt{|f(ax)|^2+f(x^*ax)^2}:f\in\mathfrak{S}_a(\mathfrak{A})\right\}.
\]
The set $DV_a(x)$ is convex \cite{dehghani2025roberts}, although, unlike the classical algebraic Davis--Wielandt shell, it need not be compact. These concepts generalize the Davis--Wielandt shell associated with operators on semi-Hilbertian spaces. It is worth mentioning that for $x\in\mathfrak{A}_a$, the $a$-Davis--Wielandt radius is given by
\[
dv_a(x)=\sup\left\{\sqrt{|f(ax)|^2+f(ax^{\sharp_a}x)^2}:f\in\mathfrak{S}_a(\mathfrak{A})\right\}.
\]

In recent years, considerable attention has also been devoted to the spatial $A$-Davis--Wielandt shell of operators on $\mathscr{B}(\mathscr{H})$. If $\mathscr{B}(\mathscr{H})^+$ denotes the cone of positive (semi-definite) operators, then every $A\in\mathscr{B}(\mathscr{H})^+$ induces a positive semi-definite sesquilinear form $\langle\cdot,\cdot\rangle_A:\mathscr{H}\times\mathscr{H}\to\mathbb{C}$ given by $\langle x,y\rangle_A=\langle Ax,y\rangle$ for all $x,y\in\mathscr{H}$. The corresponding seminorm is $\|x\|_A=\sqrt{\langle x,x\rangle_A}=\|A^{\frac{1}{2}}x\|$. The existence of an $A$-adjoint of an operator $T\in\mathscr{B}(\mathscr{H})$ is equivalent to the solvability of the operator equation $AX=T^*A$. By Douglas's theorem, an operator $T\in\mathscr{B}(\mathscr{H})$ admits an $A$-adjoint if and only if $R(T^*A)\subseteq R(A)$. Consequently, $\mathscr{B}_A(\mathscr{H})=\{T\in\mathscr{B}(\mathscr{H}):R(T^*A)\subseteq R(A)\}$. If $T\in\mathscr{B}_A(\mathscr{H})$, then the unique solution of the equation $AX=T^*A$ is denoted by $T^{\sharp_A}$. Furthermore, if $T^{\sharp_A}\in\mathscr{B}_A(\mathscr{H})$, then $(T^{\sharp_A})^{\sharp_A}=P_ATP_A$ and $((T^{\sharp_A})^{\sharp_A})^{\sharp_A}=T^{\sharp_A}$, where $P_A$ denotes the orthogonal projection onto $\overline{R(A)}$. It is also known that if $\mathfrak{A}=\mathscr{B}(\mathscr{H})$, then $\mathfrak{A}^A=\mathscr{B}_{A^{1/2}}(\mathscr{H})$, where $\mathscr{B}_{A^{1/2}}(\mathscr{H})=\{T\in\mathscr{B}(\mathscr{H}):\exists\ \lambda>0,\|Tx\|_A\leq\lambda\|x\|_A,\ \forall x\in\mathscr{H}\}$. Moreover, for every $T\in\mathscr{B}_{A^{1/2}}(\mathscr{H})$, $\|T\|_A=\sup\{\|Tx\|_A:x\in\mathscr{H},\|x\|_A=1\}<\infty$. Within this framework, Feki et al. \cite{feki2020davis} introduced the \emph{$A$-Davis--Wielandt shell} of an operator $T\in\mathscr{B}(\mathscr{H})$ as $DW_A(T)=\{(\langle Tx,x\rangle_A,\langle Tx,Tx\rangle_A):x\in\mathscr{H},\|x\|_A=1\}.$ Subsequently, numerous upper bounds and inequalities have been established for the corresponding $A$-Davis--Wielandt radius $dw_A(T)$; see, for instance, \cite{bhanja2021,feki2024inequalities,guesba2024a-davis}.

Motivated by these developments and employing the framework of Orlicz functions, in this article we establish several new upper bounds for the $a$-Davis--Wielandt radius as well as the Davis--Wielandt radius of elements in a unital $C^*$-algebra. The rest of this paper is organized as follows: In Section~\ref{sec2}, we recall some relevant preliminary results regarding Orlicz functions. Section~\ref{sec3} presents new upper bounds for the algebraic $a$-Davis--Wielandt radius obtained via Orlicz functions. Section~\ref{sec4} establishes several new upper bounds for the Davis--Wielandt radius by employing polar decomposition and the Moore--Penrose inverse in conjunction with Orlicz functions. Our results not only extend the existing theory, but also unify and recover a variety of known inequalities as special cases.
   \section{Preliminaries}\label{sec2}
   
		 This work develops several upper bounds for the $a$-Davis--Wielandt radius and the Davis--Wielandt radius of elements in a $C^*$-algebra through the use of Orlicz functions. For suitable choices of Orlicz functions, a number of known as well as new inequalities are obtained as special cases.
\begin{definition}
A mapping $\phi:[0,\infty)\to[0,\infty)$ is said to be an \emph{Orlicz function} provided that $\phi$ is continuous, convex, non-decreasing, $\phi(0)=0$, and $\phi(u)\to\infty$ as $u\to\infty$ (see \cite{lindenstraussclassical}).
\end{definition}
A few commonly used examples of Orlicz functions are given by
\begin{eqnarray*}
\phi(u)=u^r,\ r\geq1, \quad
\phi(u)=u^r\log(1+u),\ r\geq1, \quad
\phi(u)=e^{u^r}-1,\ r\geq1.
\end{eqnarray*}
We call an Orlicz function $\phi$ non-degenerate if it takes strictly positive values on $(0,\infty)$. If there exists some $u>0$ for which $\phi(u)=0$, then $\phi$ is referred to as degenerate. In the present work, only non-degenerate Orlicz functions will be considered.

An Orlicz function $\phi$ is said to possess the sub-multiplicative property if $\phi(uv)\leq \phi(u)\phi(v)$ for every $u,v\geq0$. It also follows from the defining properties that $\phi(\lambda u)\leq\lambda\phi(u)$ whenever $\lambda\in[0,1]$ and $u\geq0$. Moreover, $\phi$ has an integral representation of the form $\phi(u)=\int_0^u p(m)\,dm$, in which $p$ is non-decreasing and satisfies $p(0)=0$, $p(m)>0$ for $m>0$, and $\lim_{m\to\infty}p(m)=\infty$. In the special situation where $\phi(u)$ is equivalent to $u$, the aforementioned assumptions on $p$ may be omitted. Define the right inverse of $p$ by $q(n)=\sup\{m:p(m)\leq n\}$, $n\geq0$. The function $\psi(v)=\int_0^v q(n)\,dn$ is then referred to as the complementary Orlicz function of $\phi$. For the particular choice $\phi(u)=\frac{u^p}{p}$, $u\geq0$, $p>1$, the associated complementary function takes the form $\psi(u)=\frac{u^q}{q}$, where $\frac{1}{p}+\frac{1}{q}=1$.

For later use, we record several inequalities associated with Orlicz functions. We first recall the Hermite--Hadamard inequality \cite{hadamard1893etude}. Suppose that $\phi:I\subseteq\mathbb{R}\to\mathbb{R}$ is convex and $u,v\in I$ satisfy $u<v$. Then
\[
\phi\left(\frac{u+v}{2}\right)
\leq
\int_0^1 \phi(tu+(1-t)v)\,dt
\leq
\frac{\phi(u)+\phi(v)}{2}.
\]
\begin{lemma}[Young's inequality {\cite{lindenstraussclassical}}]\label{young}
Suppose that $\phi$ and $\psi$ form a pair of complementary Orlicz functions. The following assertions hold:
\begin{enumerate}
\item[(i)] for all $u,v\geq0$, $uv\leq \phi(u)+\psi(v)$.
\item[(ii)] for every $u\geq0$, $up(u)=\phi(u)+\psi(p(u))$ (equality condition).
\end{enumerate}
\end{lemma}


\begin{lemma}\cite{maji2022orlicz}
For an Orlicz function $\phi$ and non-negative numbers $u_i$, $i=1,2,\dots,n$, the following inequality is valid:
\[
\phi\left(\frac{1}{n}\sum_{i=1}^n u_i\right)
\leq
\frac{1}{n}\sum_{i=1}^n \phi(u_i).
\]
\end{lemma}

\section{Upper bounds of algebraic $a$-Davis--Wielandt radius}\label{sec3}

Let $f\in\mathfrak{S}_a(\mathfrak{A})$. Then by \cite[Lemma~2.4]{bourhim2021a-numerical} there exists a Hilbert space $\mathscr{H}$, a representation $\pi:\mathfrak{A}\to\mathscr{B}(\mathscr{H})$ and a unique cyclic vector $\xi\in\mathscr{H}$ such that $\langle\pi(a)\xi,\xi\rangle=1$ and $f(x)=\langle\pi(x)\xi,\xi\rangle$ for $x\in\mathfrak{A}$. Now, set $A=\pi(a)$. Throughout the paper, $A$ stands for the operator $\pi(a)$. Since $a$ is a positive element in $\mathfrak{A}$ and $\pi$ is a $*$-homomorphism, $A\geq0$. From the definition of semi-inner product, we have
\[f(ax)=\langle\pi(ax)\xi,\xi\rangle=\langle\pi(a)\pi(x)\xi,\xi\rangle=\langle A\pi(x)\xi.\xi\rangle=\langle\pi(x)\xi,\xi\rangle_A.\] Furthermore, $\|\xi\|_A^2=\langle A\xi,\xi\rangle=\langle\pi(a)\xi,\xi\rangle=1.$
\begin{lemma}\label{pi} 
   Let $x\in\mathfrak{A}_a$ and $x^{\sharp_a}$ be an $a$-adjoint of $x$. Then the following statements hold:
   \begin{enumerate}
   \item [(i)] $x^{\sharp_a}\in\mathfrak{A}_a$.
   \item [(ii)] $x^{\sharp_a}x$, $xx^{\sharp_a}\in\mathfrak{A}_a$.
  \item[(iii)] If $\pi:\mathfrak{A}\to\mathscr{B}(\mathscr{H})$ is a representation and $A=\pi(a)$, then $\pi(x)\in\mathscr{B}_A(\mathscr{H})$. Moreover, $\pi(x)^{\sharp_A}=\pi(x^{\sharp_a})$.
  \item [(iv)] $a(x^{\sharp_a}x)^{\sharp_a}x^{\sharp_a}x=a(x^{\sharp_a}x)^2=ax^{\sharp_a}x(x^{\sharp_a}x)^{\sharp_a}$.
   \end{enumerate}
\end{lemma}
\begin{proof}
(i)  We have $ax^{\sharp_a}=x^*a$. This implies $(x^{\sharp_a})^*a=(x^*a)^*=ax$. Thus, the element $x\in\mathfrak{A}$ satisfies the equation $ay=(x^{\sharp_a})^*a$ for some $y\in\mathfrak{A}$, so we obtain $x^{\sharp_a}\in\mathfrak{A}_a$.

(ii) Let $(x^{\sharp_a})^{\sharp_a}$ be an $a$-adjoint of $x^{\sharp_a}$. Let $y=x^{\sharp_a}x$ and $z=x^{\sharp_a}(x^{\sharp_a})^{\sharp_a}$. Consequently, 
\[az=ax^{\sharp_a}(x^{\sharp_a})^{\sharp_a}=x^*a(x^{\sharp_a})^{\sharp_a}=x^*(x^{\sharp_a})^*a=(x^{\sharp_a}x)^*a=y^*a.\] Therefore, $z$ is an $a$-adjoint of $y$, which implies $x^{\sharp_a}x\in\mathfrak{A}_a$. In a similar way, we can show that $xx^{\sharp_a}\in\mathfrak{A}_a$.

   (iii) Let $x\in\mathfrak{A}_a$ and $x^{\sharp_a}$ be an $a$-adjoint of $x$. Then we have
    \[ax^{\sharp_a}=x^*a\implies \pi(a)\pi(x^{\sharp_a})=\pi(x)^*\pi(a)\implies A\pi(x^{\sharp_a})=\pi(x)^*A.\] Therefore, the operator $\pi(x^{\sharp_a})$ is an $A$-adjoint of $\pi(x)$. Thus, $\pi(x)\in\mathscr{B}_A(\mathscr{H})$. Also, by the uniqueness of the $A$-adjoint of the operator on $\mathscr{B}_A(\mathscr{H})$, we get $\pi(x)^{\sharp_A}=\pi(x^{\sharp_a})$.

    (iv) Observe that \[a(x^{\sharp_a}x)^{\sharp_a}x^{\sharp_a}x=(x^{\sharp_a}x)^*ax^{\sharp_a}x=x^*(x^{\sharp_a})^*ax^{\sharp_a}x=x^*axx^{\sharp_a}x=a(x^{\sharp_a}x)^2.\] Similarly, we can verify $ax^{\sharp_a}x(x^{\sharp_a}x)^{\sharp_a}=a(x^{\sharp_a}x)^2$. 
\end{proof}
\begin{lemma}\cite[Corollary~1]{alomari2020generalized}\label{positive}
    Let $T\in\mathscr{B}_A(\mathscr{H})$ be such that $T$ is positive and let $x\in\mathscr{H}$ be an $A$-unit vector. Then
    \[\langle Tx,x\rangle_A^r\leq\langle T^rx,x\rangle_A,\qquad r\geq1.\]
\end{lemma}
With the help of the Lemma~\ref{positive}, one can easily verify the next lemma.
\begin{lemma}\label{sp}
    Let $x\in\mathfrak{A}_a$ be a positive element and $f\in\mathfrak{S}_a(\mathfrak{A})$. Then 
    \[f(ax)^r\leq f(ax^r),\qquad r\geq1.\]
\end{lemma}
Before proceeding further, we recall the following lemma.
\begin{lemma}\cite[Lemma~2.2]{mahapatra2024upper}\label{c-s_inq}
   Let $f\in\mathfrak{S}_a(\mathfrak{A})$ and $x,y\in\mathfrak{A}$. Then 
   \[|f(x^*ay)|^2\leq f(x^*ax)f(y^*ay).\]
\end{lemma}
For $x,y\in\mathfrak{A}_a$ as a consequence of Lemma~\ref{c-s_inq}, we have the following results. Putting $x=\mathbf{1}$, we get
\begin{eqnarray}\label{cons1}
    |f(ay)|^2\leq f(a)f(ay^{\sharp_a}y)=f(ay^{\sharp_a}y).
\end{eqnarray} If we put $y=x^{\sharp_a}$ in inequality~\eqref{cons1}, we get
\begin{eqnarray}\label{cons2}
    |f(ax^{\sharp_a})|^2\leq f(a(x^{\sharp_a})^{\sharp_a}x^{\sharp_a}).
\end{eqnarray} Let $(x^{\sharp_a})^{\sharp_a}$ be an $a$-adjoint of $x^{\sharp_a}$. Using $a(x^{\sharp_a})^{\sharp_a}=(x^{\sharp_a})^*a=ax$, the inequality~\eqref{cons2} reduces to 
\begin{eqnarray}\label{cons3}
     |f(ax^{\sharp_a})|^2\leq f(axx^{\sharp_a}).
\end{eqnarray}
In this section, our aim is to obtain estimations for the algebraic $a$-Davis--Wielandt radius using the Orlicz functions. We start this section with the following lemma.
\begin{lemma}\label{inequ5}\cite[Lemma~2.4]{qiao2022a-norm}
Let $T\in \mathscr{B}_A(\mathscr{H})$. Then for any $x,y\in\mathscr{H}$ with $\|x\|_A=\|y\|_A=1$, we have 
\[
  |\langle Tx,y\rangle_A|^2\leq\sqrt{\langle T^{\sharp_A}Tx,x\rangle_A}\sqrt{\langle TT^{\sharp_A}y,y\rangle_A}.
\]
\end{lemma}
As a consequence of Lemma~\ref{inequ5}, we have the following result.
\begin{lemma}\label{nl1}
    Let $x\in\mathfrak{A}_a$, $x^{\sharp_a}$ be an $a$-adjoint of $x$ and $f\in\mathfrak{S}_a(\mathfrak{A})$. Then
    \[|f(ax)|^2\leq \sqrt{f(ax^{\sharp_a}x)}\sqrt{f(axx^{\sharp_a})}.\]
\end{lemma}
\begin{proof}
    For $x\in\mathfrak{A}_a$, $x^{\sharp_a}$ is an $a$-adjoint of $x$, $f\in\mathfrak{S}_a(\mathfrak{A})$ and $\xi\in\mathscr{H}$ with $\|\xi\|_A=1$, we have
    \begin{eqnarray*}
        |f(ax)|^2=|\langle\pi(x)\xi,\xi\rangle_A|^2
        &\leq&\sqrt{\langle\pi(x)^{\sharp_A}\pi(x)\xi,\xi\rangle_A}\sqrt{\langle\pi(x)\pi(x)^{\sharp_A}\xi,\xi\rangle_A}\quad\text{(by Lemma~\ref{inequ5})}\\
&=&\sqrt{\langle\pi(x^{\sharp_a})\pi(x)\xi,\xi\rangle_A}\sqrt{\langle\pi(x)\pi(x^{\sharp_a})\xi,\xi\rangle_A}\quad \text{(by Lemma~\ref{pi}(iii))}\\
        &=&\sqrt{f(ax^{\sharp_a}x)}\sqrt{f(axx^{\sharp_a})}.
    \end{eqnarray*}
\end{proof}
We begin with the first estimate for the algebraic $a$-Davis--Wielandt radius $dv_a(\cdot)$.
\begin{theorem}
	Let $x\in \mathfrak{A}_a$, $x^{\sharp_a}$ be an $a$-adjoint of $x$ and $f\in \mathfrak{S}_a(\mathfrak{A})$. If $\phi$ is an Orlicz function, then
	\begin{eqnarray*}
		\phi(dv_a^2(x))&\leq& \frac{1}{2}\phi\left(v_a(x^{\sharp_a}x+xx^{\sharp_a})-\inf_{f\in\mathfrak{S}_a(\mathfrak{A})}\left(\sqrt{f(ax^{\sharp_a}x)}-\sqrt{f(axx^{\sharp_a})}\right)^2\right)
        +\frac{1}{2}\phi\left(v_a(2(x^{\sharp_a}x)^2)\right).
	\end{eqnarray*}
\end{theorem}
\begin{proof} Let $x\in \mathfrak{A}_a$, $x^{\sharp_a}$ be an $a$-adjoint of $x$ and $f\in \mathfrak{S}_a(\mathfrak{A})$. Then 
	\begin{eqnarray*}
			 &&|f(ax)|^2+f(ax^{\sharp_a}x)^2\\  
			& \leq& \sqrt{f(ax^{\sharp_a}x)}\sqrt{f(axx^{\sharp_a})}+\sqrt{f(a(x^{\sharp_a}x)^{\sharp_a}x^{\sharp_a}x)}\sqrt{f(ax^{\sharp_a}x(x^{\sharp_a}x)^{\sharp_a})}\quad \text{(from Lemma~\ref{nl1})}\\
			&=& \frac{1}{2}\left(f(ax^{\sharp_a}x+axx^{\sharp_a})-\left(\sqrt{f(ax^{\sharp_a}x)}-\sqrt{f(axx^{\sharp_a})}\right)^2\right)+\frac{1}{2}f(2a(x^{\sharp_a}x)^2)\quad \text{(by Lemma~\ref{pi}(iv))}.
	\end{eqnarray*}
  Since $\phi$ is a non-decreasing and convex function, we have
  \begin{eqnarray}\label{4.1}
  		 &&\phi(|f(ax)|^2+f(ax^{\sharp_a}x)^2)\nonumber\\
  		&\leq& \displaystyle\int_{0}^{1}\phi\left(t\left(f(a(x^{\sharp_a}x+xx^{\sharp_a}))-\left(\sqrt{f(ax^{\sharp_a}x)}-\sqrt{f(axx^{\sharp_a})}\right)^2\right)+(1-t)f(2a(x^{\sharp_a}x)^2)\right)\nonumber\\
  		&\leq& \frac{1}{2}\phi\left(f(a(x^{\sharp_a}x+xx^{\sharp_a}))-\left(\sqrt{f(ax^{\sharp_a}x)}-\sqrt{f(axx^{\sharp_a})}\right)^2\right)+
        \frac{1}{2}\phi\left(f(2a(x^{\sharp_a}x)^2)\right).		
  \end{eqnarray}
  Taking the supremum over $f\in \mathfrak{S}_a(\mathfrak{A})$, we get the desired inequality.
\end{proof}
\begin{corollary}\label{cor.4.2}
    Let $x\in \mathfrak{A}_a$, $x^{\sharp_a}$ be an $a$-adjoint of $x$ and $f\in \mathfrak{S}_a(\mathfrak{A})$. Then
    \[dv_a^2(x)\leq\frac{1}{2}v_a\left(x^{\sharp_a}x+xx^{\sharp_a}+2(x^{\sharp_a}x)^2\right)-\frac{1}{2}\inf_{f\in\mathfrak{S}_a(\mathfrak{A})}\left(\sqrt{f(ax^{\sharp_a}x)}-\sqrt{f(axx^{\sharp_a})}\right)^2.\]
\end{corollary}
\begin{proof}
    Choosing $\phi(t)=t$, $t\geq0$ in inequality~\eqref{4.1} and taking the supremum over $f\in\mathfrak{S}_a(\mathfrak{A})$, the required inequality holds.
\end{proof}
\begin{remark}
    If $\mathfrak{A}$ is taken as $\mathscr{B}(\mathscr{H})$ and $a=\mathbf{1}$, then Corollary~\ref{cor.4.2} becomes \cite[Theorem~2.2]{zamani2020some}.
\end{remark}
To prove our next result, we require the following lemma.
\begin{lemma}\label{thm1}\cite[Theorem~2.16]{mahapatra2024upper}
	Let $x\in \mathfrak{A}_a$ and $\phi$ be an Orlicz function. Then
	\[
		\phi(v_a^2(x))\leq\frac{1}{2}\phi(v_a(x^2))+\frac{1}{2}\phi\left( \frac{\|xx^{\sharp_a}+x^{\sharp_a}x\|_a}{2}\right).
	\]
\end{lemma}
\begin{theorem}\label{new_res}
	Let $x\in \mathfrak{A}_a$, $x^{\sharp_a}$ be an $a$-adjoint of $x$ and $\phi$ be an Orlicz function. Then
	\begin{eqnarray*}
			\phi(dv_a^2(x))&\leq& \frac{1}{4}\left[\phi(v_a(y^2))+\phi\left(\frac{\|yy^{\sharp_a}+y^{\sharp_a}y\|_a}{2}\right)+\phi(v_a(z^2)) 
		 +\phi\left(\frac{\|zz^{\sharp_a}+z^{\sharp_a}z\|_a}{2}\right)\right], 
	\end{eqnarray*} where $y=x^{\sharp_a}x+x$ and $z=x^{\sharp_a}x-x$.
\end{theorem}
\begin{proof} Let $x\in \mathfrak{A}_a$, $x^{\sharp_a}$ be an $a$-adjoint of $x$ and $f\in \mathfrak{S}_a(\mathfrak{A})$. Then
\begin{eqnarray*}
          \phi(|f(ax)|^2+|f(ax^{\sharp_a}x)|^2) 
		& =&\phi\left(\frac{1}{2}(|f(ax^{\sharp_a}x)+f(ax)|^2)+\frac{1}{2}(|f(ax^{\sharp_a}x)-f(ax)|^2)\right) \\
		& \leq&\displaystyle\int_{0}^{1}\phi\left(t(|f(ax^{\sharp_a}x)+f(ax)|^2)+(1-t)(|f(ax^{\sharp_a}x)-f(ax)|^2)\right)dt \\
		&\leq& \frac{1}{2}\phi(|f(ax^{\sharp_a}x+ax)|^2)+\frac{1}{2}\phi(|f(ax^{\sharp_a}x-ax)|^2) \quad\text{(since $\phi$ is convex}) \\
        &\leq& \frac{1}{2}\phi(v_a^2(x^{\sharp_a}x+x))+\frac{1}{2}\phi(v_a^2(x^{\sharp_a}x-x))\\
		& \leq& \frac{1}{4}\bigg[\phi(v_a(y^2))+\phi\left(\frac{\|yy^{\sharp_a}+y^{\sharp_a}y\|_a}{2}\right)+\phi(v_a(z^2)) 
		 +\phi\left(\frac{\|zz^{\sharp_a}+z^{\sharp_a}z\|_a}{2}\right)\bigg], \\ && \text{(using Lemma~\ref{thm1})}
\end{eqnarray*} where $y=x^{\sharp_a}x+x\in\mathfrak{A}_a$ and $z=x^{\sharp_a}x-x\in\mathfrak{A}_a$ as $\mathfrak{A}_a$ forms a subalgebra. Furthermore, taking the supremum over $f\in \mathfrak{S}_a(\mathfrak{A})$, we get the required inequality.
\end{proof}
It is easy to observe the following corollary.
\begin{corollary}
    Let $x\in\mathfrak{A}_a$, $x^{\sharp_a}$ be an $a$-adjoint of $x$. Then for any $r\geq1$,
    \[dv_a^{2r}(x)\leq\frac{1}{4}\left[v_a^r(y^2)+\frac{1}{2^r}\|yy^{\sharp_a}+y^{\sharp_a}y\|^r_a+v_a^r(z^2)+\frac{1}{2^r}\|zz^{\sharp_a}+z^{\sharp_a}z\|^r_a\right],\]  where $y=x^{\sharp_a}x+x$ and $z=x^{\sharp_a}x-x$.
\end{corollary}
\begin{proof}
   If we consider $\phi(t)=t^r$, $r\geq1$, $t\geq0$ in Theorem~\ref{new_res}, then we obtain the needed result.
\end{proof}
We now state another extension of the Cauchy–Schwarz inequality. By \cite[p.125]{dragomir2007advances} one can verify the next lemma.
\begin{lemma}\label{inequ6}
	For any $x, y, z \in \mathscr{H}$, \[|\langle x, y \rangle_A|^2+|\langle x, z \rangle_A|^2\leq \|x\|_A^2(\max\{\|y\|_A^2, \|z\|_A^2\}+\sqrt{2}|\langle y, z \rangle_A|).\]
\end{lemma}
The next result provides an upper bound for the algebraic $a$-Davis--Wielandt radius.
\begin{theorem}
	Let $x\in \mathfrak{A}_a$, $x^{\sharp_a}$ be an $a$-adjoint of $x$ and $\phi$ be an Orlicz function. Then\[
		\phi(dv_a^2(x))\leq \frac{1}{4}\phi\left(\|2x^{\sharp_a}x+2(x^{\sharp_a}x)^2\|_a\right)+\frac{1}{4}\phi\left(\|2x^{\sharp_a}x-2(x^{\sharp_a}x)^2\|_a\right)+\frac{1}{2}\phi(v_a(2\sqrt{2}x^{\sharp_a}x^2)).
	\]
\end{theorem}
\begin{proof} Let $x\in \mathfrak{A}_a$, $x^{\sharp_a}$ be an $a$-adjoint of $x$, $f\in \mathfrak{S}_a(\mathfrak{A})$ and $\xi\in\mathscr{H}$ with $\|\xi\|_A=1$. Then we get
	\begin{eqnarray*}
			 &&|f(ax)|^2+|f(ax^{\sharp_a}x)|^2 \\
			& = & |\langle\xi, \pi(x)\xi\rangle_A|^2+|\langle\xi, \pi(x^{\sharp_a}x)\xi\rangle_A|^2\\
			& \leq &\|\xi\|_A^2\left(\max\{\|\pi(x)\xi\|_A^2, \| \pi(x^{\sharp_a}x)\xi\|_A^2\}+\sqrt{2}|\langle\pi(x)\xi, \pi(x^{\sharp_a}x)\xi\rangle_A|\right) \quad\text{(by Lemma~\ref{inequ6})}\\
			& = &\max \{\langle\pi(x^{\sharp_a}x)\xi, \xi\rangle_A, \langle\pi((x^{\sharp_a}x)^{\sharp_a}(x^{\sharp_a}x))\xi, \xi\rangle_A\}+\sqrt{2}|\langle\pi((x^{\sharp_a})^2x)\xi, \xi\rangle_A|\quad \text{(by Lemma~\ref{pi}(iii))}\\
			& \leq& \max \{f(ax^{\sharp_a}x), f(a(x^{\sharp_a}x)^{\sharp_a}x^{\sharp_a}x)\}+\sqrt{2}|f(a((x^{\sharp_a})^2 x)^{\sharp_a})| \quad \text{(as $|f(ax)|=|f(ax^{\sharp_a})|$)}\\
			& =& \frac{1}{2}\left[f(ax^{\sharp_a}x+a(x^{\sharp_a}x)^2)+|f(ax^{\sharp_a}x-a(x^{\sharp_a}x)^2)|\right]+\frac{1}{2}|f(2\sqrt{2}ax^{\sharp_a}x^2)|\quad\text{(from Lemma~\ref{pi}(iv))}.
	\end{eqnarray*}
	As $\phi$ is a non-decreasing and convex function, we have
	\begin{eqnarray}
			&& \phi(|f(ax)|^2+|f(ax^{\sharp_a}x)|^2)\nonumber \\
			& \leq& \phi\left(\frac{1}{2}\left[f(ax^{\sharp_a}x+a(x^{\sharp_a}x)^2)+|f(ax^{\sharp_a}x-a(x^{\sharp_a}x)^2)|\right]+\frac{1}{2}|f(2\sqrt{2}ax^{\sharp_a}x^2)|\right) \nonumber\\
			& \leq& \displaystyle\int_{0}^{1}\phi\left(t(f(ax^{\sharp_a}x+a(x^{\sharp_a}x)^2)+|f(ax^{\sharp_a}x-a(x^{\sharp_a}x)^2)|)+(1-t)(|f(2\sqrt{2}ax^{\sharp_a}x^2)|)\right)dt\nonumber \\
			& \leq& \frac{1}{2}\phi(f(ax^{\sharp_a}x+a(x^{\sharp_a}x)^2)+|f(ax^{\sharp_a}x-a(x^{\sharp_a}x)^2)|)+\frac{1}{2}\phi(|f(2\sqrt{2}ax^{\sharp_a}x^2)|)\label{100} \\
			& \leq& \frac{1}{4}\phi(f(2ax^{\sharp_a}x+2a(x^{\sharp_a}x)^2))+\frac{1}{4}\phi(|f(2ax^{\sharp_a}x-2a(x^{\sharp_a}x)^2)|)
            +\frac{1}{2}\phi(|f(2\sqrt{2}ax^{\sharp_a}x^2)|)\label{101}.
	\end{eqnarray}
	If we take the supremum over $f\in \mathfrak{S}_a(\mathfrak{A})$, then the desired result holds. 
\end{proof}
\begin{corollary}\label{cor.4.8}
    Let $x\in \mathfrak{A}_a$, $x^{\sharp_a}$ be an $a$-adjoint of $x$. Then
	\[
			 dv_a^2(x)\leq\max\{\|x\|_a^2,\|x\|_a^4\}+\sqrt{2}v_a(x^{\sharp_a}x^2).\]
\end{corollary}
\begin{proof}
  Choosing $\phi(t)=t$, $t\geq0$ in inequality~\eqref{101} we get
  \begin{eqnarray*}
      |f(ax)|^2+f(ax^{\sharp_a}x)^2&\leq& \frac{1}{2}f(ax^{\sharp_a}x+a(x^{\sharp_a}x)^2)+\frac{1}{2}|f(ax^{\sharp_a}x-a(x^{\sharp_a}x)^2|+\sqrt{2}|f(ax^{\sharp_a}x^2)|\\
      &=&\max\{f(ax^{\sharp_a}x),f(a(x^{\sharp_a}x)^2)\}+\sqrt{2}|f(ax^{\sharp_a}x^2)|\\
      &\leq&\max\{\|x^{\sharp_a}x\|_a,\|x^{\sharp_a}x\|_a^2\}+\sqrt{2}v_a(x^{\sharp_a}x^2)\\
      &=& \max\{\|x\|_a^2,\|x\|_a^4\}+\sqrt{2}v_a(x^{\sharp_a}x^2).
  \end{eqnarray*} Taking the supremum over $f\in \mathfrak{S}_a(\mathfrak{A})$, we obtain the needed inequality.
\end{proof}
\begin{remark}
    If we consider $\mathfrak{A}=\mathscr{B}(\mathscr{H})$ and $a=\mathbf{1}$ in Corollary~\ref{cor.4.8}, then we get \cite[Theorem~2.13]{zamani2020some}.
\end{remark}
In addition, an upper bound of $dv_a(\cdot)$ is as follows.
\begin{corollary}\label{thm.4.10}
	Let $x\in \mathfrak{A}_a$, $x^{\sharp_a}$ be an $a$-adjoint of $x$. Then
	\[
		dv_a^2(x)\leq \frac{1}{2}(v_a(x^{\sharp_a}x+(x^{\sharp_a}x)^2)+v_a(x^{\sharp_a}x-(x^{\sharp_a}x)^2)+\sqrt{2}v_a(x^{\sharp_a}x^2). 
	\]
\end{corollary}
\begin{proof} Let $x\in \mathfrak{A}_a$, $x^{\sharp_a}$ be an $a$-adjoint of $x$ and $f\in \mathfrak{S}_a(\mathfrak{A})$. Then from inequality~\eqref{100} choosing $\phi(t)=t$, $t\geq0$, we have
			\[ |f(ax)|^2+f(ax^{\sharp_a}x)^2\leq\frac{1}{2}(f(ax^{\sharp_a}x+a(x^{\sharp_a}x)^2)+|f(ax^{\sharp_a}x-a(x^{\sharp_a}x)^2|)+\sqrt{2}|f(ax^{\sharp_a}x^2)|.\]
	Taking the supremum over $f\in \mathfrak{S}_a(\mathfrak{A})$, we get the required result.
\end{proof}
\begin{remark}
    If we fix $\phi(t)=t$, $t\geq0$, $a=\mathbf{1}$ and $\mathfrak{A}=\mathscr{B}(\mathscr{H})$ in Corollary~\ref{thm.4.10}, then we acquire \cite[Theorem~2.14]{zamani2020some}.
\end{remark}
For $x\in\mathfrak{A}$, the $a$-Crawford number of an element $x$, denoted by $\mathfrak{C}_a(x)$, is defined as
\[\mathfrak{C}_a(x)=\inf\{|f(ax)|:f\in\mathfrak{S}_a(\mathfrak{A})\}.\]
The next theorem also provides an upper bound of the algebraic $a$-Davis-Wielandt radius.
\begin{theorem}\label{nthm1}
		Let $x\in \mathfrak{A}_a$, $x^{\sharp_a}$ be an $a$-adjoint of $x$ and $\phi$ be an Orlicz function. Then
			\[
			\phi(dv_a^2(x))\leq \frac{1}{2}\phi\left(4\|x\|_a^2-2\sqrt{\mathfrak{C}_a(x^{\sharp_a}x)\mathfrak{C}_a(xx^{\sharp_a})}\right)+\frac{1}{2}\phi\left(2\|x\|_a^4\right).
		\]
\end{theorem}
\begin{proof} Let $x\in \mathfrak{A}_a$, $x^{\sharp_a}$ be an $a$-adjoint of $x$ and $f\in \mathfrak{S}_a(\mathfrak{A})$. Then we have
	\begin{eqnarray*}
			 &&|f(ax)|^2+f(ax^{\sharp_a}x)^2 \\
			& \leq &\sqrt{f(ax^{\sharp_a}x)}\sqrt{f(axx^{\sharp_a})}+\sqrt{f(a(x^{\sharp_a}x)^{\sharp_a}x^{\sharp_a}x)}\sqrt{f(ax^{\sharp_a}x(x^{\sharp_a}x)^{\sharp_a})}\quad \text{(using Lemma~\ref{nl1})} \\
			& \leq& \frac{1}{2}\left(f(ax^{\sharp_a}x)+f(axx^{\sharp_a})\right)+f(a(x^{\sharp_a}x)^2)\quad\text{(using Lemma~\ref{pi}(iv))}  \\
            &=& \frac{1}{2}\left[\left(\sqrt{f(ax^{\sharp_a}x)}+\sqrt{f(axx^{\sharp_a})}\right)^2-2\sqrt{f(ax^{\sharp_a}x)}\sqrt{f(axx^{\sharp_a})}\right]+\frac{1}{2}f(2a(x^{\sharp_a}x)^2). \\
	\end{eqnarray*}
	As $\phi$ a is non-decreasing and convex, we have
	\begin{eqnarray*}\label{d1}
			&& \phi(|f(ax)|^2+f(ax^{\sharp_a}x)^2) \nonumber\\
			& \leq& \displaystyle\int_{0}^{1}\phi\left(t\left(\left(\sqrt{f(ax^{\sharp_a}x)}+\sqrt{f(axx^{\sharp_a})}\right)^2-2\sqrt{f(ax^{\sharp_a}x)}\sqrt{f(axx^{\sharp_a})}\right)+(1-t)\left(f(2a(x^{\sharp_a}x)^2)\right)\right)dt \nonumber\\
			& \leq &\frac{1}{2}\phi\left(\left(\sqrt{\|x^{\sharp_a}x\|_a}+\sqrt{\|xx^{\sharp_a}\|_a}\right)^2-2\sqrt{\mathfrak{C}_a(x^{\sharp_a}x)\mathfrak{C}_a(xx^{\sharp_a})}\right)+\frac{1}{2}\phi\left(2\|x^{\sharp_a}x\|_a^2\right)\\
            &=&\frac{1}{2}\phi\left(4\|x\|_a^2-2\sqrt{\mathfrak{C}_a(x^{\sharp_a}x)\mathfrak{C}_a(xx^{\sharp_a})}\right)+\frac{1}{2}\phi\left(2\|x\|_a^4\right).
	\end{eqnarray*}
		Taking the supremum over $f\in \mathfrak{S}_a(\mathfrak{A})$, we get the desired inequality.
\end{proof}
In general, for an element $x\in\mathfrak{A}_a$, neither $x^{\sharp_a}x$ nor $xx^{\sharp_a}$ is necessarily positive. However, both are $a$-self-adjoint and $a$-positive. Motivated by these observations, the $a$-absolute value of the element $x$ is defined by $|x|_a^2=ax^{\sharp_a}x$,
which is positive element. In addition, we can write $|x|_a=(ax^{\sharp_a}x)^{\frac{1}{2}}$. This property is called the uniqueness of the square root of positive elements. 

We establish below an upper estimate for $dv_a(\cdot)$ in terms of non-negative continuous functions. The proof makes use of the following lemmas.
\begin{lemma}\label{cts}
    Let $x\in\mathfrak{A}_a$ and $x^{\sharp_a}$ be an $a$-adjoint of $x$. Suppose that $ax=xa$. Then $a$ commutes with every continuous function of $|x|_a$. 
\end{lemma}
\begin{proof}
    Since $ax=xa$, we have $ax^*=x^*a$. Also, $x\in\mathfrak{A}_a$ gives $ax^{\sharp_a}=x^*a$. Now,
    \[a|x|_a^2=a(ax^{\sharp_a}x)=a(x^*ax)=a(x^*xa)=|x|_a^2a.\] As $|x|_a$ is the unique positive square root of $|x|_a^2$, it follows that $a|x|_a=|x|_aa.$ Hence, $a$ commutes with every polynomial in $|x|_a$ and therefore, by the Stone–Weierstrass theorem, with every continuous function of $|x|_a$.
\end{proof}
\begin{lemma}\label{inequ7}\cite{sattari2015some} (Power-Young inequality)
	Let $u, v \geq 0$ and $\alpha, \beta>1$ be such that $\frac{1}{\alpha}+\frac{1}{\beta}=1$. Then 
	\[
		uv\leq \frac{1}{\alpha}u^\alpha+\frac{1}{\beta}v^\beta.
	\]
\end{lemma}
\begin{lemma}\label{lemma2}\cite[Lemma~2.1]{rashid2026}
 Let $T \in\mathscr{B}_A(\mathscr{H})$ be such that $AT=TA$. If $f_1,g_1$ are two non-negative continuous functions on $[0, \infty)$ such that $f_1(t)g_1(t)=t$, for all $t\in[0, \infty)$. Then 
 \[
 	|\langle Tx, y\rangle_A|\leq \|f_1(|T|_A)x\|_A\|g_1(|T^{\sharp_A}|_A)y\|_A,\quad \mathrm{for}\ \mathrm{all}\ x, y \in\mathscr{H}.
 \]
\end{lemma}
\begin{lemma}\label{lem1}
	Let $x\in \mathfrak{A}_a$ with $ax=xa$, $x^{\sharp_a}$ be an $a$-adjoint of $x$ and $f\in \mathfrak{S}_a(\mathfrak{A})$. If $f_1, g_1$ are two non-negative continuous functions on $[0, \infty)$ such that $f_1(t)g_1(t)=t$, for all $t\in[0, \infty)$. Then \[
		|f(ax)|\leq \sqrt{f(af_1^2(|x|_a))}\sqrt{f(ag_1^2(|x^{\sharp_a}|_a))}.
	\]
\end{lemma}
\begin{proof} Let $x\in \mathfrak{A}_a$ with $ax=xa$, $x^{\sharp_a}$ be an $a$-adjoint of $x$, $f\in \mathfrak{S}_a(\mathfrak{A})$ and $\xi\in\mathscr{H}$ with $\|\xi\|_A=1$. Then
	\begin{eqnarray*}
			 |f(ax)| 
			& =& |\langle\pi(x)\xi, \xi\rangle_A|\\
			& \leq& \left\|f_1\left(|\pi(x)|_A\right)\xi\right\|_A\left\|g_1\left(|\pi(x)^{\sharp_A}|_A\right)\xi\right\|_A \quad\text{(from Lemma~\ref{lemma2})}\\
			& =& \sqrt{\langle A f_1(|\pi(x)|_A\xi,f_1(|\pi(x)|_A\xi\rangle}\sqrt{\langle A g_1(|\pi(x)^{\sharp_A}|_A\xi,g_1(|\pi(x)^{\sharp_A}|_A\xi\rangle} \\
			& = & \sqrt{\langle f_1(|\pi(x)|_A)^*Af_1(|\pi(x)|_A)\xi,\xi\rangle} \sqrt{\langle g_1(|\pi(x)^{\sharp_A}|_A)^*Ag_1(|\pi(x)^{\sharp_A}|_A)\xi,\xi\rangle}\\
            &=& \sqrt{\langle f_1(\pi(|x|_a))^*\pi(a)f_1(\pi(|x|_a))\xi,\xi\rangle} \sqrt{\langle g_1(\pi(|x^{\sharp_a}|_a))^*\pi(a)g_1(\pi(|x^{\sharp_a}|_a)\xi,\xi\rangle}\\
            &&\text{(by Lemma~\ref{pi}(iii) and using $|\pi(x)|_A=\pi(|x|_a)$)}\\
			& =& \sqrt{\langle\pi(f_1(|x|_a)af_1(|x|_a))\xi,\xi\rangle}\sqrt{\langle\pi(g_1(|x^{\sharp_a}|_a)ag_1(|x^{\sharp_a}|_a))\xi,\xi\rangle} \quad\text{(using continuous function calculus)}\\
			& = &\sqrt{f(af_1^2(|x|_a))}\sqrt{f(ag_1^2(|x^{\sharp_a}|_a))}\quad\text{(by Lemma~\ref{cts})}.
	\end{eqnarray*}
\end{proof}
\begin{theorem}\label{thm.4.15}
		Let $x\in \mathfrak{A}_a$ with $ax=xa$, $x^{\sharp_a}$ be an $a$-adjoint of $x$. Assume that $f_1, g_1, f_2, g_2$ are  non-negative continuous functions on $[0, \infty)$ such that $f_1(t)g_1(t)=t$ and $f_2(t)g_2(t)=t$, for all $t\in[0, \infty)$. Also, $\phi$ be an Orlicz function and $\alpha_1, \alpha_2, \beta_1, \beta_2\geq 1$ such that $\frac{1}{\alpha_1}+\frac{1}{\beta_1}=1$ and $\frac{1}{\alpha_2}+\frac{1}{\beta_2}=1$. Then
		\begin{eqnarray*}
			\phi(dv_a^2(x))
            &\leq& \phi\left(\frac{1}{\alpha_1}\|f_1^{2\alpha_1}(|x|_a)\|_a\right)+\phi\left(\frac{1}{\beta_1}\|g_1^{2\beta_1}(|x^{\sharp_a}|_a)\|_a\right)+\phi\left(\frac{1}{\alpha_2}\|f_2^{2\alpha_2}(|x^{\sharp_a}x|_a)\|_a\right)\\
            &&+\phi\left(\frac{1}{\beta_2}\|g_2^{2\beta_2}(|x^{\sharp_a}x|_a)\|_a\right).
		\end{eqnarray*}
\end{theorem}
\begin{proof} Let $x\in \mathfrak{A}_a$ with $ax=xa$, $x^{\sharp_a}$ be an $a$-adjoint of $x$ and $f\in \mathfrak{S}_a(\mathfrak{A})$. Then
	\begin{eqnarray*}
			 |f(ax)|^2 
			& \leq& f(af_1^2(|x|_a))f(ag_1^2(|x^{\sharp_a}|_a)) \quad\text{(from Lemma~\ref{lem1})} \\
			& \leq& \frac{1}{\alpha_1}(f(af_1^2(|x|_a)))^{\alpha_1}+\frac{1}{\beta_1}(f(ag_1^2(|x^{\sharp_a}|_a)))^{\beta_1} \quad\text{(from Lemma~\ref{inequ7})} \\
			& \leq& f\left(\frac{1}{\alpha_1}\left(af_1^{2\alpha_1}(|x|_a)\right)+\frac{1}{\beta_1}\left(ag_1^{2\beta_1}(|x^{\sharp_a}|_a)\right)\right)\quad\text{(by Lemma~\ref{sp})}.
	\end{eqnarray*}
	Similarly, 
	\[
	f(ax^{\sharp_a}x)^2\leq f\left(\frac{1}{\alpha_2}\left(af_1^{2\alpha_2}(|x^{\sharp_a}x|_a)\right)+\frac{1}{\beta_2}\left(ag_1^{2\beta_2}(|x^{\sharp_a}x|_a)\right)\right).
	\]
	Therefore,
	\[
			|f(ax)|^2+f(ax^{\sharp_a}x)^2
			 \leq f\left(\frac{1}{\alpha_1}af_1^{2\alpha_1}(|x|_a)+\frac{1}{\beta_1}ag_1^{2\beta_1}(|x^{\sharp_a}|_a)\right)+f\left(\frac{1}{\alpha_2}af_1^{2\alpha_2}(|x^{\sharp_a}x|_a)+\frac{1}{\beta_2}ag_1^{2\beta_2}(|x^{\sharp_a}x|_a)\right). 
		\]
	Hence, by the non-negativity and convexity property of $\phi$, we get
	 \begin{eqnarray*}
		&&\phi\left(\frac{|f(ax)|^2+f(ax^{\sharp_a}x)^2}{2}\right) \nonumber\\
			& \leq& \displaystyle\int_{0}^{1}\phi\left(tf\left(\frac{1}{\alpha_1}af_1^{2\alpha_1}(|x|_a)+\frac{1}{\beta_1}ag_1^{2\beta_1}(|x^{\sharp_a}|_a)\right)+(1-t)f\left(\frac{1}{\alpha_2}af_2^{2\alpha_2}(|x^{\sharp_a}x|_a)+\frac{1}{\beta_2}ag_2^{2\beta_2}(|x^{\sharp_a}x|_a)\right)\right)dt \nonumber\\
			& \leq& \frac{1}{2}\phi\left(f\left(\frac{1}{\alpha_1}af_1^{2\alpha_1}(|x|_a)+\frac{1}{\beta_1}ag_1^{2\beta_1}(|x^{\sharp_a}|_a)\right)\right)+\frac{1}{2}\phi\left(f\left(\frac{1}{\alpha_2}af_2^{2\alpha_2}(|x^{\sharp_a}x|_a)+\frac{1}{\beta_2}ag_2^{2\beta_2}(|x^{\sharp_a}x|_a)\right)\right) \nonumber\\
			& \leq& \frac{1}{2}\bigg(\phi\left(\frac{1}{\alpha_1}f(af_1^{2\alpha_1}(|x|_a))\right)+\phi\left(\frac{1}{\beta_1}f(ag_1^{2\beta_1}(|x^{\sharp_a}|_a))\right)+\phi\left(\frac{1}{\alpha_2}f(af_2^{2\alpha_2}(|x^{\sharp_a}x|_a))\right)\\
            &&+\phi\left(\frac{1}{\beta_2}f(ag_2^{2\beta_2}(|x^{\sharp_a}x|_a))\right)\bigg)\nonumber. 
            \end{eqnarray*}
	Taking the supremum over $f\in \mathfrak{S}_a(\mathfrak{A})$, we get the desired result.
\end{proof}
\begin{corollary}\label{new thm_1}
   Let $x\in \mathfrak{A}_a$ with $ax=xa$, $x^{\sharp_a}$ be an $a$-adjoint of $x$. Assume that $f_1, g_1, f_2, g_2$ are  non-negative continuous functions on $[0, \infty)$ such that $f_1(t)g_1(t)=t$ and $f_2(t)g_2(t)=t$, for all $t\in[0, \infty)$. Also, $\alpha_1, \alpha_2, \beta_1, \beta_2\geq 1$ such that $\frac{1}{\alpha_1}+\frac{1}{\beta_1}=1$ and $\frac{1}{\alpha_2}+\frac{1}{\beta_2}=1$. Then
    \[dv_a^2(x)\leq \left\|\frac{1}{\alpha_1}f_1^{2\alpha_1}(|x|_a)+\frac{1}{\beta_1}g_1^{2\beta_1}(|x^{\sharp_a}|_a)+\frac{1}{\alpha_2}f_2^{2\alpha_2}(|x^{\sharp_a}x|_a)+\frac{1}{\beta_2}g_2^{2\beta_2}(|x^{\sharp_a}x|_a)\right\|_a.\]
\end{corollary}
\begin{proof}
    Taking into account $\phi(t)=t$, $t\geq0$ in the last inequality of Theorem~\ref{thm.4.15}, we get
    \begin{eqnarray*}
        |f(ax)|^2+f(ax^{\sharp_a}x)^2
            &\leq& f\left(a\left(\frac{1}{\alpha_1}f_1^{2\alpha_1}(|x|_a)+\frac{1}{\beta_1}g_1^{2\beta_1}(|x^{\sharp_a}|_a)+\frac{1}{\alpha_2}f_2^{2\alpha_2}(|x^{\sharp_a}x|_a)+\frac{1}{\beta_2}g_2^{2\beta_2}(|x^{\sharp_a}x|_a)\right)\right)\\
            &=&\left\|\frac{1}{\alpha_1}f_1^{2\alpha_1}(|x|_a)+\frac{1}{\beta_1}g_1^{2\beta_1}(|x^{\sharp_a}|_a)+\frac{1}{\alpha_2}f_2^{2\alpha_2}(|x^{\sharp_a}x|_a)+\frac{1}{\beta_2}g_2^{2\beta_2}(|x^{\sharp_a}x|_a)\right\|_a.
    \end{eqnarray*}
    Taking the supremum over $f\in\mathfrak{S}_a(\mathfrak{A})$, the result follows.
\end{proof}
\begin{remark}
    If we choose $a=\mathbf{1}$ and $\mathfrak{A}=\mathscr{B}(\mathscr{H})$, then Corollary~\ref{new thm_1} becomes \cite[Theorem~2.3]{bhunia2021bounds}.
\end{remark}

Next, we prove the following theorem using Lemma~\ref{lem1}.
\begin{theorem}
		  Let $x\in \mathfrak{A}_a$ with $ax=xa$, $x^{\sharp_a}$ be an $a$-adjoint of $x$. Suppose $f_1,g_1$ are two non-negative continuous functions on $[0, \infty)$ such that $f_1(t)g_1(t)=t$, for all $t\in[0, \infty)$. Then for an Orlicz function $\phi$ and any $\alpha\in[0, 1]$,
		\[
			\phi(dv_a^2(x))\leq \frac{\alpha}{2}\phi(v_a(q))+\frac{1-\alpha}{2}\phi(v_a(s))+\frac{1}{2}\phi(v_a(f_1^2(|x^{\sharp_a}x|_a^2)+g_1^2(|x^{\sharp_a}x|_a^2)))
		,\]
			where $q=f_1^2(|x^{\sharp_a}x|_a)+g_1^2(|x^{\sharp_a}x|_a)$ and $s=f_1^2(|xx^{\sharp_a}|_a)+g_1^2(|xx^{\sharp_a}|_a)$.
\end{theorem}
\begin{proof}Let $x\in \mathfrak{A}_a$ with $ax=xa$, $x^{\sharp_a}$ be an $a$-adjoint of $x$ and $f\in \mathfrak{S}_a(\mathfrak{A})$. For any $\alpha\in[0, 1]$, we have
	\begin{eqnarray*}
			&& |f(ax)|^2+f(ax^{\sharp_a}x)^2 \\
			& =&\alpha|f(ax)|^2+(1-\alpha)|f(ax^{\sharp_a})|^2+f(ax^{\sharp_a}x)^2 \quad\text{(since $|f(ax)|=|f(ax^{\sharp_a})|)$} \\
            &\leq& \alpha f(ax^{\sharp_a}x)+(1-\alpha)f(axx^{\sharp_a})+f(a(x^{\sharp_a}x)^2)\quad\text{(by inequalities~\eqref{cons1}, \eqref{cons3} and Lemma~\ref{nl1})}\\
			& \leq& \alpha\left[\sqrt{f(af_1^2(|x^{\sharp_a}x|_a))}\sqrt{f(ag_1^2(|x^{\sharp_a}x|_a))}\right]+(1-\alpha)\left[\sqrt{f(af_1^2(|xx^{\sharp_a}|_a))}\sqrt{fa(g_1^2(|xx^{\sharp_a}|_a))}\right]\\
			&& +\sqrt{f(af_1^2(|x^{\sharp_a}x|_a^2))}\sqrt{f(ag_1^2(|x^{\sharp_a}x|_a^2))} \quad\text{(by Lemma~\ref{lem1})} \\
			& \leq& \frac{\alpha}{2}\left(f(af_1^2(|x^{\sharp_a}x|_a))+f(ag_1^2(|x^{\sharp_a}x|_a))\right)+\frac{1-\alpha}{2}\left(f(af_1^2(|xx^{\sharp_a}|_a))+f(ag_1^2(|xx^{\sharp_a}|_a))\right)\\
			&& +\frac{1}{2}\left(f(af_1^2(|x^{\sharp_a}x|_a^2))+f(ag_1^2(|x^{\sharp_a}x|_a^2))\right) \\
			& =& \frac{1}{2}f(a\alpha q+a(1-\alpha)s)+\frac{1}{2}\left(f(af_1^2(|x^{\sharp_a}x|_a^2))+f(ag_1^2(|x^{\sharp_a}x|_a^2))\right),
	\end{eqnarray*}
	where $q=f_1^2(|x^{\sharp_a}x|_a)+g_1^2(|x^{\sharp_a}x|_a)$ and $s=f_1^2(|xx^{\sharp_a}|_a)+g_1^2(|xx^{\sharp_a}|_a)$. From the non-decreasing and convex properties of $\phi$, we get
	\begin{eqnarray}\label{4.3}
			 &&\phi(|f(ax)|^2+f(ax^{\sharp_a}x)^2) \nonumber\\
			& \leq& \displaystyle\int_{0}^{1}\phi\left(t(f(a(\alpha q+(1-\alpha)s)))+(1-t)\left(f(a(f_1^2(|x^{\sharp_a}x|_a^2)+g_1^2(|x^{\sharp_a}x|_a^2))\right)\right)\, dt \nonumber\\
			& \leq &\frac{1}{2}\phi\left(f(a(\alpha q+(1-\alpha)s))\right)+\frac{1}{2}\phi\left(f(a(f_1^2(|x^{\sharp_a}x|_a^2)+g_1^2(|x^{\sharp_a}x|_a^2)))\right)\nonumber\\
            &\leq& \frac{\alpha}{2}\phi(f(a(q)))+\frac{1-\alpha}{2}\phi(f(a(s)))+\frac{1}{2}\phi\left(f(a(f_1^2(|x^{\sharp_a}x|_a^2)+g_1^2(|x^{\sharp_a}x|_a^2)))\right).
	\end{eqnarray}
		Taking the supremum over $f\in \mathfrak{S}_a(\mathfrak{A})$, we get the required result.
		\end{proof}
        \begin{corollary}\label{cor.4.19}
            	  Let $x\in \mathfrak{A}_a$ with $ax=xa$, $x^{\sharp_a}$ be an $a$-adjoint of $x$. Suppose $f_1,g_1$ are two non-negative continuous functions on $[0, \infty)$ such that $f_1(t)g_1(t)=t$, for all $t\in[0, \infty)$. Then for any $\alpha\in[0,1]$,
                \[dv_a^2(x)\leq\frac{1}{2}\left\|\alpha q+(1-\alpha)s+f_1^2(|x^{\sharp_a}x|_a^2)+g_1^2(|x^{\sharp_a}x|_a^2)\right\|_a,\] where $q=f_1^2(|x^{\sharp_a}x|_a)+g_1^2(|x^{\sharp_a}x|_a)$ and $s=f_1^2(|xx^{\sharp_a}|_a)+g_1^2(|xx^{\sharp_a}|_a)$.
        \end{corollary}
        \begin {proof}
        From inequality~\eqref{4.3} choosing $\phi(t)=t$, $t\geq0$, we get
        \begin{eqnarray*}
            |f(ax)|^2+f(ax^{\sharp_a}x)^2&\leq&\frac{1}{2}f\left(a\left(\alpha q+(1-\alpha)s+f_1^2(|x^{\sharp_a}x|_a^2)+g_1^2(|x^{\sharp_a}x|_a^2)\right)\right)\\
            &=&\frac{1}{2}\left\|\alpha q+(1-\alpha)s+f_1^2(|x^{\sharp_a}x|_a^2)+g_1^2(|x^{\sharp_a}x|_a^2)\right\|_a,
        \end{eqnarray*}where $q=f_1^2(|x^{\sharp_a}x|_a)+g_1^2(|x^{\sharp_a}x|_a)$ and $s=f_1^2(|xx^{\sharp_a}|_a)+g_1^2(|xx^{\sharp_a}|_a)$. Taking the supremum over $f\in\mathfrak{S}_a(\mathfrak{A})$, we obtain the desired inequality.
        \end{proof}
 We now derive the following theorem.
\begin{theorem}
		Let $x\in \mathfrak{A}_a$ with $ax=xa$, $x^{\sharp_a}$ be an $a$-adjoint of $x$. Suppose $f_1, g_1$ are two non-negative continuous functions on $[0, \infty)$ such that $f_1(t)g_1(t)=t$, for all $t\in[0, \infty)$. If $\phi$ is an Orlicz function and $\alpha\in [0,1]$, then
		\begin{eqnarray*}
				\phi(dv_a^2(x))
				& \leq& \frac{\alpha}{2}\left[\phi\left(v_a\left(\frac{(f_1^2(|x|_a)+g_1^2(|x^{\sharp_a}|_a))^2}{2}\right)\right)+\phi\left(v_a\left(\frac{(f_1^2(|x^{\sharp_a}x|_a)+g_1^2(|x^{\sharp_a}x|_a))^2}{2}\right)\right)\right]\\
				&& +(1-\alpha)\phi\left(v_a(xx^{\sharp_a}+(x^{\sharp_a}x)^2)\right).
		\end{eqnarray*}
\end{theorem}
\begin{proof} Let $x\in \mathfrak{A}_a$ with $ax=xa$, $x^{\sharp_a}$ be an $a$-adjoint of $x$ and $f\in \mathfrak{S}_a(\mathfrak{A})$. For any $\alpha\in [0, 1]$, we have
	\begin{eqnarray*}
		&& |f(ax)|^2+f(ax^{\sharp_a}x)^2 \\
		& =	&\alpha|f(ax)|^2+(1-\alpha)|f(ax)|^2+\alpha f(ax^{\sharp_a}x)^2+(1-\alpha)f(ax^{\sharp_a}x)^2 \\
		& \leq& \alpha\left[f(af_1^2(|x|_a))f(ag_1^2(|x^{\sharp_a}|_a))+ f(af_1^2(|x^{\sharp_a}x|_a))f(g_1^2(|x^{\sharp_a}x|_a))\right] 
		 +(1-\alpha)\left(|f(ax^{\sharp_a})|^2+f(ax^{\sharp_a}x)^2\right)\\ 
		 &&\text{(by Lemma~\ref{lem1} and also $|f(ax)|=|f(ax^{\sharp_a})|$)} \\
		& \leq& \alpha\left[f(af_1^2(|x|_a))f(ag_1^2(|x^{\sharp_a}|_a))+ f(af_1^2(|x^{\sharp_a}x|_a))f(g_1^2(|x^{\sharp_a}x|_a))\right] 
		 +(1-\alpha)\left(f(axx^{\sharp_a})+f(a(x^{\sharp_a}x)^2)\right)\\
         &&\quad\text{(using inequality~\eqref{cons3} and Lemma~\ref{nl1})} \\
		& \leq& \alpha\left[f\left(a\left(\frac{f_1^2(|x|_a)+g_1^2(|x^{\sharp_a}|_a)}{2}\right)^2\right)+ f\left(a\left(\frac{f_1^2(|x^{\sharp_a}x|_a)+g_1^2(|x^{\sharp_a}x|_a)}{2}\right)^2\right)\right]
        +(1-\alpha)f(a(xx^{\sharp_a}+(x^{\sharp_a}x)^2)). 
	\end{eqnarray*}
	As $\phi$ is non-decreasing and convex, we get
	\begin{eqnarray}\label{thm.4.19}
			&& \phi(|f(ax)|^2+f(ax^{\sharp_a}x)^2)\nonumber \\
			& \leq& \frac{\alpha}{2}\phi\left(f\left(a\left(\frac{(f_1^2(|x|_a)+g_1^2(|x^{\sharp_a}|_a))^2}{2}\right)\right)+f\left(a\left(\frac{(f_1^2(|x^{\sharp_a}x|_a)+g_1^2(|x^{\sharp_a}x|_a))^2}{2}\right)\right)\right)\nonumber\\ 
			&&+(1-\alpha)\phi\left(f(a(xx^{\sharp_a}+(x^{\sharp_a}x)^2))\right) \nonumber\\
            &\leq& \frac{\alpha}{2}\left[\phi\left(f\left(a\left(\frac{(f_1^2(|x|_a)+g_1^2(|x^{\sharp_a}|_a))^2}{2}\right)\right)\right)+\phi\left(f\left(a\left(\frac{(f_1^2(|x^{\sharp_a}x|_a)+g_1^2(|x^{\sharp_a}x|_a))^2}{2}\right)\right)\right)\right]\nonumber\\
            &&+(1-\alpha)\phi\left(f(a(xx^{\sharp_a}+(x^{\sharp_a}x)^2))\right).
	\end{eqnarray}
		Taking the supremum over $f\in \mathfrak{S}_a(\mathfrak{A})$ of the above inequality, the required result holds.
\end{proof}
\begin{corollary}\label{cor.4.19}
    Let $x\in \mathfrak{A}_a$ with $ax=xa$, $x^{\sharp_a}$ be an $a$-adjoint of $x$. Suppose $f_1, g_1$ are two non-negative continuous functions on $[0, \infty)$ such that $f_1(t)g_1(t)=t$, for all $t\in[0, \infty)$. Then for any $\alpha\in [0,1]$,
    \[dv_a^2(x)\leq\left\|\alpha(q^2+s^2)+(1-\alpha)(xx^{\sharp_a}+(x^{\sharp_a}x)^2)\right\|_a,\] where $q=\frac{1}{2}(f_1^2(|x|_a)+g_1^2(|x^{\sharp_a}|_a))$ and $s=\frac{1}{2}(f_1^2(|x^{\sharp_a}x|_a)+g_1^2(|x^{\sharp_a}x|_a))$.
\end{corollary}
\begin{proof}
    In inequality~\eqref{thm.4.19} if we choose $\phi(t)=t$, $t\geq0$, then we have \begin{eqnarray*}
    &&|f(ax)|^2+f(ax^{\sharp_a}x)^2\\
    &\leq& f\bigg(a\bigg(\frac{\alpha}{4}(f_1^2(|x|_a)+g_1^2(|x^{\sharp_a}|_a))^2+\frac{\alpha}{4}(f_1^2(|x^{\sharp_a}x|_a)+g_1^2(|x^{\sharp_a}x|_a))^2
    +(1-\alpha)(xx^{\sharp_a}+(x^{\sharp_a}x)^2)\bigg)\bigg)\\
    &=&\left\|\alpha(q^2+s^2)+(1-\alpha)(xx^{\sharp_a}+(x^{\sharp_a}x)^2)\right\|_a,\end{eqnarray*} where $q=\frac{1}{2}(f_1^2(|x|_a)+g_1^2(|x^{\sharp_a}|_a))$ and $s=\frac{1}{2}(f_1^2(|x^{\sharp_a}x|_a)+g_1^2(|x^{\sharp_a}x|_a))$. Taking the supremum over $f\in\mathfrak{S}_a(\mathfrak{A})$, yields the desired result.
\end{proof}
\begin{remark}
      If we take $a=\mathbf{1}$ and $\mathfrak{A}=\mathscr{B}(\mathscr{H})$, then Corollary~\ref{cor.4.19} becomes \cite[Theorem~2.19, inequality~(2.19) for $r=1$]{bhunia2021new}.
\end{remark}

    \section{Upper bounds of algebraic Davis--Wielandt radius}\label{sec4}

    We begin this section to derive an estimate of $dv(\cdot)$ in terms of non-negative continuous functions. In the sequel to prove Theorem~\ref{new thm}, we need the following lemmas. In \cite[Lemma~2.10]{mahapatra2024upper} by putting $a=\mathbf{1}$ we state Lemma~\ref{buzano}.
\begin{lemma}\label{buzano}
	Let $x,y \in \mathfrak{A}$ and $f \in \mathfrak{S}(\mathfrak{A})$. Then
	\[
		|f(x)||f(y)|\leq \frac{1}{2}\sqrt{f(x^*x)}\sqrt{f(yy^*)}+\frac{1}{2}|f(yx)|.
	\]
\end{lemma}
\begin{lemma}\label{lm2}
	  Let $x\in \mathfrak{A}$ and $f\in \mathfrak{S}(\mathfrak{A})$. Suppose $f_1,g_1$ are two non-negative continuous functions on $[0, \infty)$ such that $f_1(t)g_1(t)=t$, for all $t\in[0, \infty)$. If $\psi_1$ and $\psi_2$ are complementary Orlicz functions of $\phi_1$ and $\phi_2$, respectively, then 
	\[
		|f(x)|^2 
		 \leq  \frac{1}{2} |f(x^2)|+\frac{1}{4}\bigg[\phi_1\left(\sqrt{f(f_1^2(|x|^2))}\right)+\psi_1\left(\sqrt{f(g_1^2(|x|^2))}\right)+\phi_2\left(\sqrt{f(f_1^2(|x^*|^2))}\right) 
	+\psi_2\left(\sqrt{f(g_1^2(|x^*|^2))}\right)\bigg].
	\]
\end{lemma}
\begin{proof} Let $x\in \mathfrak{A}$ and $f\in \mathfrak{S}(\mathfrak{A})$. Then
	\begin{eqnarray*}
			 &&|f(x)|^2 \\
			& \leq& \frac{1}{2} |f(x^2)|+\frac{1}{2}\sqrt{f(x^*x)}\sqrt{f(xx^*)} \quad\text{(from Lemma~\ref{buzano})} \\
			& \leq& \frac{1}{2} |f(x^2)|+\frac{1}{4}\left(f(x^*x)+f(xx^*)\right)  \\
			& \leq&  \frac{1}{2} |f(x^2)|+\frac{1}{4}\left[\sqrt{f(f_1^2(|x|^2))}\sqrt{f(g_1^2(|x|^2))}
+\sqrt{f(f_1^2(|x^*|^2))}\sqrt{f(g_1^2(|x^*|^2))}\right]\quad\text{(by Lemma~\ref{lem1})} \\
			& \leq&  \frac{1}{2} |f(x^2)|+\frac{1}{4}\left[\phi_1\left(\sqrt{f(f_1^2(|x|^2))}\right)+\psi_1\left(\sqrt{f(g_1^2(|x|^2))}\right)+\phi_2\left(\sqrt{f(f_1^2(|x^*|^2))}\right) 
	+\psi_2\left(\sqrt{f(g_1^2(|x^*|^2))}\right)\right]\\
    && \text{(using Lemma~\ref{young})}.
	\end{eqnarray*}
\end{proof}

If we take $\phi_1(t)= \phi_2(t)= \psi_1(t)= \psi_2(t)=\frac{t^2}{2}$, then from Lemma~\ref{lm2} we get
\begin{eqnarray}\label{0.3}
		 |f(x)|^2 
		 \leq \frac{1}{2} |f(x^2)|+\frac{1}{8}\left[f\left(f_1^2(|x|^2)+g_1^2(|x|^2)+f_1^2(|x^*|^2)+g_1^2(|x^*|^2)\right)\right].
\end{eqnarray}

\begin{theorem}\label{new thm}
	Let $x\in \mathfrak{A}$. Suppose $f_1, g_1$ are two non-negative continuous functions on $[0, \infty)$ such that $f_1(t)g_1(t)=t$, for all $t\in[0, \infty)$. If $\phi$ is an Orlicz function and $\alpha \in [0, 1]$, then
	\begin{eqnarray*}
		\phi(dv^2(x))
		& \leq & \frac{\alpha}{2}\bigg[\phi\left(v(x^2)+v(|x|^4)\right)+\phi\bigg(\frac{1}{4}(v(f_1^2(|x|^2)+g_1^2(|x|^2)+f_1^2(|x^*|^2)+g_1^2(|x^*|^2)))+\frac{1}{2} \nonumber\\
			&& (v(f_1^2(|x|^4)+g_1^2(|x|^4)))\bigg)\bigg]+(1-\alpha)\phi\left(v(|x^*|^2+|x|^4)\right)\nonumber.
	\end{eqnarray*}
\end{theorem}
\begin{proof} Let $x\in \mathfrak{A}$ and $f\in \mathfrak{S}(\mathfrak{A})$. For any $\alpha\in [0,1]$ and using the fact $|f(x)|=|f(x^*)|$, we get
	\begin{eqnarray}\label{thm.4.16}
    && \phi(|f(x)|^2+f(x^*x)^2) \nonumber\\
			& =& \phi\left(\alpha|f(x)|^2+(1-\alpha)|f(x)|^2+\alpha f(|x|^2)^2+(1-\alpha)f(|x|^2)^2\right) \nonumber\\
			& \leq& \phi\bigg(\frac{\alpha}{2} |f(x^2)|+\frac{\alpha}{8}(f(f_1^2(|x|^2)+g_1^2(|x|^2)+f_1^2(|x^*|^2)+g_1^2(|x^*|^2)))+\frac{\alpha}{2}f(|x|^4)+\frac{\alpha}{4} \nonumber\\
			&& (f(f_1^2(|x|^4)+g_1^2(|x|^4)))+(1-\alpha)(|f(x^*)|^2+f(|x|^2)^2))\bigg) \quad\text{(by inequality~\eqref{0.3})} \nonumber\\
			&\leq& \phi\bigg(\frac{\alpha}{2} |f(x^2)|+\frac{\alpha}{8}(f(f_1^2(|x|^2)+g_1^2(|x|^2)+f_1^2(|x^*|^2)+g_1^2(|x^*|^2)))+\frac{\alpha}{2}f(|x|^4)+\frac{\alpha}{4}\nonumber\\
			&& (f(f_1^2(|x|^4)+g_1^2(|x|^4)))+(1-\alpha)f(xx^*+|x|^4)\bigg)
            \quad\text{(using inequality~\eqref{cons3} and Lemma~\ref{sp})} \nonumber\\
			& \leq& \alpha\phi\bigg(\frac{1}{2} |f(x^2)|+\frac{1}{8}(f(f_1^2(|x|^2)+g_1^2(|x|^2)+f_1^2(|x^*|^2)+g_1^2(|x^*|^2)))+\frac{1}{2}f(|x|^4)+\frac{1}{4}\nonumber\\
            &&(f(f_1^2(|x|^4)+g_1^2(|x|^4)))\bigg)+(1-\alpha)\phi\left(f(|x^*|^2+|x|^4)\right) \nonumber\\
			& \leq& \frac{\alpha}{2}\bigg[\phi\left(|f(x^2)|+f(|x|^4)\right)+\phi\bigg(\frac{1}{4}(f(f_1^2(|x|^2)+g_1^2(|x|^2)+f_1^2(|x^*|^2)+g_1^2(|x^*|^2)))+\frac{1}{2} \nonumber\\
			&& (f(f_1^2(|x|^4)+g_1^2(|x|^4)))\bigg)\bigg]+(1-\alpha)\phi\left(f(|x^*|^2+|x|^4)\right).		
	\end{eqnarray}
	Taking the supremum over $f\in \mathfrak{S}(\mathfrak{A})$, we get the required result.
\end{proof} 
It is easy to observe the next corollary.
\begin{corollary}\label{cor.4.17}
  Let $x\in \mathfrak{A}$. Suppose $f_1, g_1$ are two non-negative continuous functions on $[0, \infty)$ such that $f_1(t)g_1(t)=t$, for all $t\in[0, \infty)$. Then for $\alpha\in [0,1]$,
    \[dv^2(x)\leq\frac{\alpha}{2}v(x^2)+\left\|\frac{\alpha}{8}l+(1-\alpha)|x^*|^2+\left(1-\frac{\alpha}{2}\right)|x|^4\right\|,\] where $l=f_1^2(|x|^2)+g_1^2(|x|^2)+f_1^2(|x^*|^2)+g_1^2(|x^*|^2)+2(f_1^2(|x|^4)+g_1^2(|x|^4))$.
\end{corollary}
\begin{proof}
In inequality \eqref{thm.4.16} choosing $\phi(t)=t$, $t\geq0$, we get 
\begin{eqnarray*}
    |f(x)|^2+f(x^*x)^2
    &\leq&\frac{\alpha}{2}|f(x^2)|+f\bigg(\frac{\alpha}{2}|x|^4+\frac{\alpha}{8}(f_1^2(|x|^2)+g_1^2(|x|^2)+f_1^2(|x^*|^2)+g_1^2(|x^*|^2))+\\
    &&2(f_1^2(|x|^4)+g_1^2(|x|^4)))+(1-\alpha)(|x^*|^2+|x|^4)\bigg)\\
    &\leq&\frac{\alpha}{2}v(x^2)+\left\|\frac{\alpha}{8}l+(1-\alpha)|x^*|^2+\left(1-\frac{\alpha}{2}\right)|x|^4\right\|,
\end{eqnarray*} where $l=f_1^2(|x|^2)+g_1^2(|x|^2)+f_1^2(|x^*|^2)+g_1^2(|x^*|^2)+2(f_1^2(|x|^4)+g_1^2(|x|^4))$. Taking the supremum over $f\in\mathfrak{S}(\mathfrak{A})$, we get the desired result. 
\end{proof}
\begin{remark}
    If we take $\mathfrak{A}=\mathscr{B}(\mathscr{H})$, then Corollary~\ref{cor.4.17} becomes \cite[Theorem~2.15, inequality~(2.11) for $r=1$]{bhunia2021new}.
\end{remark}
Let $x$ be an invertible element in a unital $C^*$-algebra $\mathfrak{A}$. Then there exists a unique unitary element $u$ such that $x=u|x|$, where $|x|=\sqrt{x^*x}$. Also, then $|x^*|^{\alpha}=u|x|^{\alpha}u^*$ for any $\alpha\in(0, \infty)$. We now calculate the upper bound of the algebraic Davis--Wielandt radius of the product of three elements of $\mathfrak{A}$ using the following lemmas.
\begin{lemma}\label{lem3}\cite[Lemma~2.7]{mahapatra2024upper}
	Let $x,y\in\mathfrak{A}$ and $f\in \mathfrak{S}(\mathfrak{A})$. Then
	\[
		f(y^*x^*xy)\leq \|x\|^2f(y^*y).
	\]
\end{lemma}
\begin{lemma}\label{lem4}\cite[Lemma~2.39]{mahapatra2024upper}
	Let $x,y,z\in\mathfrak{A}$ and $y$ be an invertible element in $\mathfrak{A}$. If $0\leq\alpha\leq1$ and $f\in \mathfrak{S}(\mathfrak{A})$, then
	\[
		|f(xyz)|^2\leq f(x|y^*|^{2(1-\alpha)}x^*)f(z^*|y|^{2\alpha}z).
	\]
\end{lemma}
\begin{lemma}\label{lemma5}
	Let $x,y,z\in\mathfrak{A}$ and $y$ be an invertible element in $\mathfrak{A}$. If $0\leq\alpha\leq1$, $f\in \mathfrak{S}(\mathfrak{A})$, and if $\phi_1$ and $\phi_2$ are complementary functions of $\psi_1$ and $\psi_2$, respectively, then
	\begin{eqnarray*}
		&&|f(xyz)|^2+f(z^*y^*x^*xyz)^2\\
        &\leq & \phi_1(f(x|y^*|^{2(1-\alpha)}x^*) )+\psi_1(f(z^*|y|^{2\alpha}z)) 
			 +\|x\|^4\left[\phi_2(f(z^*|y|^{2(2-\alpha)}z))+\psi_2(f(z^*|y|^{2\alpha}z))\right].
	\end{eqnarray*}
\end{lemma}
\begin{proof} From Lemma~\ref{lem4}, we have
		\[
		|f(xyz)|^2\leq f(x|y^*|^{2(1-\alpha)}x^*)f(z^*|y|^{2\alpha}z).
	\]
	Also, by Lemma~\ref{lem3}, we have
	\[
			 f((xyz)^*(xyz))^2 
			 = f(z^*y^*x^*xyz)^2 
			 \leq \|x\|^4 f(z^*y^*yz)^2.
	\]
	Since $y$ is invertible, $f(z^*y^*yz)^2 = f(z^*y^*u|y|z)^2$.
	Taking $z^*y^*=p$ and using Lemma~\ref{c-s_inq}, we get
	\begin{eqnarray*}
			 f(pu|y|z)^2 
			& = &f(pu|y|^{1-\alpha}|y|^\alpha z)^2 \\
			& \leq& f(pu|y|^{1-\alpha}|y|^{1-\alpha}u^*p^*)f(z^*|y|^\alpha|y|^\alpha z) \\
			& = &f(p|y^*|^{2(1-\alpha)}p^*)f(z^*|y|^{2\alpha}z) \\
			& =& f(z^*|y|u^*|y^*|^{2(1-\alpha)}u|y|z)f(z^*|y|^{2\alpha}z) \\
			& = &f(z^*|y||y|^{2(1-\alpha)}|y|z)f(z^*|y|^{2\alpha}z) \\
			& = & f(z^*|y|^{2(2-\alpha)}z)f(z^*|y|^{2\alpha}z).
	\end{eqnarray*}
	Hence, 
	\begin{eqnarray*}
		&&|f(xyz)|^2+f(z^*y^*x^*xyz)^2\\
        &\leq& f(x|y^*|^{2(1-\alpha)}x^*)f(z^*|y|^{2\alpha}z)+\|x\|^4f(z^*|y|^{2(2-\alpha)}z)f(z^*|y|^{2\alpha}z)\\
        & \leq &\phi_1(f(x|y^*|^{2(1-\alpha)}x^*) )+\psi_1(f(z^*|y|^{2\alpha}z)) 
				 +\|x\|^4\left[\phi_2(f(z^*|y|^{2(2-\alpha)}z))+\psi_2(f(z^*|y|^{2\alpha}z))\right]\quad\text{(by Lemma~\ref{young})}.
	\end{eqnarray*}
\end{proof}
				
\begin{theorem}
Let $x,y,z\in\mathfrak{A}$ be such that $y$ is invertible in $\mathfrak{A}$ and $\|x\|\leq1$. Also, let $0\leq\alpha\leq1$. Then for any $n\geq2$,
		\[
		dv^2(e)\leq \frac{1}{n}\left\|(x|y^*|^{2(1-\alpha)}x^*)^n +(z^*|y|^{2(2-\alpha)}z)^n\right\|  +\frac{2(n-1)}{n}\left\|z^*|y|^{2\alpha}z\right\|^{\tfrac{n}{n-1}}
	,\] where $e=xyz$.
\end{theorem}
\begin{proof}
	Let $f\in\mathfrak{S}(\mathfrak{A})$. Putting $\phi_1(t)=\phi_2(t)=\frac{t^n}{n}$ and $\psi_1(t)=\psi_2(t) =\frac{n-1}{n}t^{\tfrac{n}{n-1}}$, for $t\geq0,\ n\geq2$ in Lemma~\ref{lemma5} and taking $\|x\|\leq 1$, we have 
	\begin{eqnarray*}
			&&	|f(xyz)|^2+f(z^*y^*x^*xyz)^2 \\
				& \leq& \frac{1}{n}\left[f(x|y^*|^{2(1-\alpha)}x^*)^n +\|x\|^4f(z^*|y|^{2(2-\alpha)}z)^n\right] 
				 +\frac{n-1}{n}\left [f(z^*|y|^{2\alpha}z)^{\tfrac{n}{n-1}}+\|x\|^4f(z^*|y|^{2\alpha}z)^{\tfrac{n}{n-1}}\right] \\
				& \leq &\frac{1}{n}\left[f((x|y^*|^{2(1-\alpha)}x^*)^n +(z^*|y|^{2(2-\alpha)}z)^n)\right]  +\frac{2(n-1)}{n}f(z^*|y|^{2\alpha}z)^{\tfrac{n}{n-1}} \\
				& \leq& \frac{1}{n}\left\|(x|y^*|^{2(1-\alpha)}x^*)^n +(z^*|y|^{2(2-\alpha)}z)^n\right\|  +\frac{2(n-1)}{n}\left\|z^*|y|^{2\alpha}z\right\|^{\tfrac{n}{n-1}}.
	\end{eqnarray*}
	Taking the supremum over $f\in \mathfrak{S}(\mathfrak{A})$, we get
	\[
		dv^2(e)\leq \frac{1}{n}\left\|(x|y^*|^{2(1-\alpha)}x^*)^n +(z^*|y|^{2(2-\alpha)}z)^n\right\|  +\frac{2(n-1)}{n}\left\|z^*|y|^{2\alpha}z\right\|^{\tfrac{n}{n-1}}
	,\]  where $e=xyz$.
\end{proof}
     An element $x\in\mathfrak{A}$ is called \emph{regular} if there exists an element $x'\in\mathfrak{A}$ such that $x=xx'x$. This element $x'$ is called a generalized inverse of $x$. It is straightforward to verify that $xx'$ and $x'x$ are idempotents in $\mathfrak{A}$. A generalized inverse $x'$ is said to be normalized if $x=xx'x$ and $x'=x'xx'$. If $\mathfrak{A}$ is equipped with an involution $^*$, idempotents $xx'$ and $x'x$ are also required to be self-adjoint, that is, $(xx')^*=xx'$, $(x'x)^*=x'x$. In this case, $x'$ is called the \emph{Moore--Penrose inverse} of $x$ and is denoted by $x^\dagger$. In \cite{koliha2007moore}, it is proved that an element of a $C^*$-algebra is regular if and only if it is Moore--Penrose invertible. Thus, for a regular element $x\in\mathfrak{A}$, the Moore--Penrose inverse $x^\dagger$ is the unique element that satisfies
\begin{eqnarray*}
(i)\ x^\dagger=x^\dagger xx^\dagger,\ (ii)\ x=xx^\dagger x,\ (iii)\ (x^\dagger x)^*=x^\dagger x,\ (iv)\ (xx^\dagger)^*=xx^\dagger.
\end{eqnarray*}
Moreover, $x$ is Moore--Penrose invertible if and only if $x^*$ is, in which case $(x^*)^\dagger=(x^\dagger)^*$. The Moore--Penrose invertibility of $x$ implies that of $x^*x$ and $xx^*$, with $(x^*x)^\dagger=x^\dagger(x^*)^\dagger$, $(xx^*)^\dagger=(x^*)^\dagger x^\dagger$. Furthermore, if $x$ is regular, then $x^\dagger$ is also regular, and $(x^\dagger)^\dagger=x$. Note that, for any representation $\pi:\mathfrak{A}\to\mathscr{B}(\mathscr{H})$, regularity of $x$ implies that $\pi(x)$ is also regular and $\pi(x)^\dagger=\pi(x^\dagger)$. If the notation $CR(\mathscr{H})$ stands for the set of closed range operators, then it is well known that any $T\in\mathscr{B}(\mathscr{H})$ has a unique Moore-Penrose inverse if and only if $T\in CR(\mathscr{H})$.

In this section, we derive some Davis--Wielandt radius inequalities of the regular elements in $\mathfrak{A}$ using Moore-Penrose inverse and Orlicz function. To prove our results, we need the following lemmas.
\begin{lemma}\cite[Theorem~2.1]{sababheh2024numerical}\label{lm 4.11}
Let $T\in CR(\mathscr{H})$. Then for any $x, y\in\mathscr{H}$, 
\[|\langle Tx,y\rangle|^2\leq\langle|T|^2x,x\rangle\langle TT^{\dagger}y,y\rangle.\]
\end{lemma}
\begin{lemma}\label{MP}
    Let $x\in\mathfrak{A}$ be a regular element and $f\in\mathfrak{S}(\mathfrak{A})$. Then
    \[|f(x)|^2\leq f(|x|^2)f(xx^\dagger).\]
\end{lemma}
\begin{proof}
Let $x\in\mathfrak{A}$ be a regular element, $f\in\mathfrak{S}(\mathfrak{A})$ and $\xi$ be a unit vector in $\mathscr{H}$. Then
\begin{eqnarray*}
    |f(x)|^2=|\langle\pi(x)\xi,\xi\rangle|^2&\leq&\langle|\pi(x)|^2\xi,\xi\rangle\langle\pi(x)\pi(x)^\dagger\xi,\xi\rangle\quad\text{(by Lemma~\ref{lm 4.11})}\\
    &=& \langle\pi(|x|^2)\xi,\xi\rangle\langle\pi(xx^\dagger)\xi,\xi\rangle\\
    &=& f(|x|^2)f(xx^\dagger).
\end{eqnarray*}
\end{proof}
\begin{theorem}\label{regthm}
    Let $x\in\mathfrak{A}$ be a regular element. If $\psi$ is a complementary Orlicz function of $\phi$, then 
    \[dv^2(x)\leq\phi\left(\sqrt{\||x|^4+|x|^8\|}\right)+\psi\left(\sqrt{\|xx^\dagger+x^\dagger x\|}\right).\]
\end{theorem}
\begin{proof}
    Let $x\in\mathfrak{A}$ be a regular element, $f\in\mathfrak{S}(\mathfrak{A})$. Then
    \begin{eqnarray}\label{reg}
        |f(x)|^2+f(x^*x)^2\nonumber
       &\leq& f(|x|^2)f(xx^\dagger)+f(|x|^4)f(x^\dagger x)\quad\text{(by Lemma~\ref{MP})}\nonumber\\
        &\leq& \sqrt{f(|x|^2)^2+f(|x|^4)^2}\sqrt{f(xx^\dagger)^2+f(x^\dagger x)^2}\nonumber\\ &&\text{(as $pr+qs\leq\sqrt{p^2+q^2}\sqrt{r^2+s^2}$, $p,q,r,s\geq0$)}\nonumber\\
        &\leq& \sqrt{f(|x|^4)+f(|x|^8)}\sqrt{f(xx^\dagger)+f(x^\dagger x)}\quad\text{(by Lemma~\ref{sp})}\nonumber\\
        &=& \sqrt{\||x|^4+|x|^8\|}\sqrt{\|xx^\dagger+x^\dagger x\|}\\
        &\leq& \phi\left(\sqrt{\||x|^4+|x|^8\|}\right)+\psi\left(\sqrt{\|xx^\dagger+x^\dagger x\|}\right)\nonumber\quad\text{(by Lemma~\ref{young})}.
    \end{eqnarray} Taking the supremum over $f\in\mathfrak{S}(\mathfrak{A})$, we obtain the required result.
\end{proof}
\begin{remark}
    Choosing $\phi(t)=\psi(t)=\frac{t^2}{2}$ in Theorem~\ref{regthm} we obtain
    \[dv^2(x)\leq\frac{1}{2}\left(\||x|^4+|x|^8\|+\|xx^\dagger+x^\dagger x\|\right).\]
\end{remark}
\begin{corollary}
 Let $x\in\mathfrak{A}$ be a regular element. Then 
 \[dv^2(x)\leq\min\{\gamma,\delta\},\] where $\gamma=\sqrt{\||x|^4+|x|^8\|}\sqrt{\|xx^\dagger+x^\dagger x\|}$ and $\delta=\sqrt{\||x^*|^4+|x^*|^8\|}\sqrt{\|xx^\dagger+x^\dagger x\|}$.
\end{corollary}
\begin{proof}
    From equation~\eqref{reg} and taking the supremum over $f\in\mathfrak{S}(\mathfrak{A})$, we get $dv^2(x)\leq\gamma$. Since $dv(x)=dv(x^*)$, we obtain $dv^2(x)\leq\delta$.
\end{proof}
\begin{theorem}
    Let $x\in\mathfrak{A}$ be a regular element and $\phi$ be an Orlicz function. Then
    \[\phi(dv^2(x))\leq\frac{1}{2}\phi(v(|x|^4+xx^\dagger))+\frac{1}{2}\phi(v(2|x|^4)).\]
\end{theorem}
\begin{proof}
     Let $x\in\mathfrak{A}$ be a regular element and $f\in\mathfrak{S}(\mathfrak{A})$. Then
     \begin{eqnarray*}
         |f(x)|^2+f(x^*x)^2&\leq& f(|x|^2)f(xx^\dagger)+f(|x|^2)^2\quad\text{(by Lemma~\ref{MP})}\\
         &\leq& \frac{f(|x|^2)^2+f(xx^\dagger)^2}{2}+f(|x|^2)^2\\
         &\leq& \frac{f(|x|^4+xx^\dagger)}{2}+\frac{1}{2}f(2|x|^4)\quad\text{(by Lemma~\ref{sp})}.
     \end{eqnarray*}
     As $\phi$ is non-decreasing and convex, we get
     \begin{eqnarray}\label{reg_eq}
       \phi(|f(x)|^2+f(x^*x)^2)&\leq& \displaystyle\int_{0}^{1}\phi\left(t(f(|x|^4+xx^\dagger))+(1-t)f(2|x|^4)\right)\,dt\nonumber\\
       &\leq&\frac{1}{2}\phi(f(|x|^4+xx^\dagger))+\frac{1}{2}\phi(f(2|x|^4))\\
       &\leq&\frac{1}{2}\phi(v(|x|^4+xx^\dagger))+\frac{1}{2}\phi(v(2|x|^4))\nonumber.
     \end{eqnarray} Taking the supremum over $f\in\mathfrak{S}(\mathfrak{A})$, we obtain the desired inequality.
\end{proof}
\begin{corollary}\label{reg_cor}
     Let $x\in\mathfrak{A}$ be a regular element. Then
     \[dv^2(x)\leq\frac{1}{2}\|3|x|^4+xx^\dagger\|.\]
\end{corollary}
\begin{proof}
  If we consider $\phi(t)=t$, $t\geq0$ in inequality~\eqref{reg_eq}, then we have
  \[|f(x)|^2+f(x^*x)^2\leq\frac{1}{2}f(|x|^4+xx^\dagger+2|x|^4)\leq\frac{1}{2}\|3|x|^4+xx^\dagger\|.\] Taking the supremum over $f\in\mathfrak{S}(\mathfrak{A})$, yields the required result.
\end{proof}
\begin{remark}
    From Corollary~\ref{reg_cor} it is easy to verify 
    \[dv^{2r}(x)\leq 2^{r-2}\|3|x|^{4r}+xx^\dagger\|,\quad\mbox{for $r\geq1$.}\]
\end{remark}
\begin{lemma}\label{lm-mp}
     Let $x\in\mathfrak{A}$ be a regular element, $f\in\mathfrak{S}(\mathfrak{A})$ and $\phi$ be an Orlicz function. Then
     \[
     \phi(|f(x)|^2+f(x^*x)^2)\leq\frac{1}{\sqrt{2}}\left|\phi(f(|x|^4+|x|^8))+i\phi(f(xx^\dagger+x^\dagger x))\right|.
     \]
\end{lemma}
\begin{proof}
     Let $x\in\mathfrak{A}$ be a regular element and $f\in\mathfrak{S}(\mathfrak{A})$. Then
     \begin{eqnarray*}
         |f(x)|^2+f(x^*x)^2&\leq& f(|x|^2)f(xx^\dagger)+f(|x|^4)f(x^\dagger x)\quad\text{(by Lemma~\ref{MP})}\\
         &\leq&\frac{1}{2}\left[f(|x|^2)^2+f(xx^\dagger)^2\right]+\frac{1}{2}\left[f(|x|^4)^2+f(x^\dagger x)^2\right]\\
         &\leq& \frac{1}{2}\left[f(|x|^4+|x|^8)\right]+\frac{1}{2}\left[f(xx^\dagger+x^\dagger x)\right]\quad\text{(by Lemma~\ref{sp})}.
     \end{eqnarray*}
     Since $\phi$ is non-decreasing and convex, we have
     \begin{eqnarray*}\label{iequ}
      \phi(|f(x)|^2+f(x^*x)^2)&\leq&\displaystyle\int_{0}^{1}\phi\left(t(f(|x|^4+|x|^8))+(1-t)f(xx^\dagger+x^\dagger x)\right)\,dt \nonumber \\
      &\leq& \frac{1}{2}\phi(f(|x|^4+|x|^8))+\frac{1}{2}\phi(f(xx^\dagger+x^\dagger x))\nonumber\\
      &\leq&\sqrt{2}\left|\frac{1}{2}\left(\phi(f(|x|^4+|x|^8))+i\phi(f(xx^\dagger+x^\dagger x))\right)\right|\quad\text{(as $|a+b|\leq\sqrt{2}|a+ib|,\ a,b\in\mathbb{R}$)}.
     \end{eqnarray*} This completes the proof.
\end{proof}
\begin{theorem}
    Let $x\in\mathfrak{A}$ be a regular element. Then
    \begin{eqnarray}\label{mpeq1}
     dv^2(x)\leq\frac{1}{\sqrt{2}}v(|x|^4+|x|^8+i(xx^\dagger+x^\dagger x))\quad\mathrm{and}
     \end{eqnarray}
      \begin{eqnarray}\label{mpeq2}
      dv^2(x)\leq\frac{1}{\sqrt{2}}v(|x^*|^4+|x^*|^8+i(xx^\dagger+x^\dagger x)).
     \end{eqnarray} 
\end{theorem}
\begin{proof}
    If we choose $\phi(t)=t$, $t\geq0$ in Lemma~\ref{lm-mp} and take the supremum over $f\in\mathfrak{S}(\mathfrak{A})$, we get the inequality~\eqref{mpeq1}. Also, proceeding similarly and using the fact that $dv(x)=dv(x^*)$ we get the inequality~\eqref{mpeq2}.
\end{proof}
We conclude this section by providing an upper bound of $dv^4(x)$ in terms of norm.
\begin{corollary}
    Let $x\in\mathfrak{A}$ be a regular element. Then
    \[dv^4(x)\leq\||x|^8+|x|^{16}+xx^\dagger+x^\dagger x\|.\]
\end{corollary}
\begin{proof}
    From inequality~\eqref{mpeq1}, we have
    \begin{eqnarray*}
        dv^4(x)
        &\leq&\frac{1}{2}v^2(|x|^4+|x|^8+i(xx^\dagger+x^\dagger x))\\
        &\leq& \frac{1}{2}\|(|x|^4+|x|^8)^2+(xx^\dagger+x^\dagger x)^2\|\quad\text{(as $v^2(b+ic)\leq\|b^2+c^2\|$, where $b,c$  self-adjoint)}\\
        &\leq&\frac{1}{2}\|2(|x|^8+|x|^{16})+2(xx^\dagger+x^\dagger x)\|\\
        &=&\||x|^8+|x|^{16}+xx^\dagger+x^\dagger x\|.
    \end{eqnarray*}
\end{proof}


\section*{Declarations}
\begin{itemize}
    \item \textit{Authors' contributions:} The authors contributed equally to this work and approved the final manuscript.
    \item \textit{Funding:} This research did not receive external funding.
    \item \textit{Data availability:} No data were used to support this study.
    \item \textit{Conflict of interest:} The authors declare that they have no conflict of interest.
\end{itemize}

\end{document}